\documentclass{amsart}
\usepackage{amsmath}
\usepackage{amscd}
\usepackage{amssymb}
\usepackage[dvips]{graphicx,color} 
\usepackage[matrix, arrow]{xy}
\usepackage{verbatim}
\newtheorem{lemma}{Lemma}[section]
\newtheorem{prop}[lemma]{Proposition}
\newtheorem{thm}[lemma]{Theorem}
\newtheorem{cor}[lemma]{Corollary}
\theoremstyle{definition}
\newtheorem{defn}[lemma]{Definition}
\newtheorem{ex}[lemma]{Example}
\newtheorem{rem}[lemma]{Remark}
\newtheorem{assump}[lemma]{Assumption}

\newcommand{\Mod}{\mathrm{Mod}}
\newcommand{\End}{\mathrm{End}}
\newcommand{\Hom}{\mathrm{Hom}}
\newcommand{\Aut}{\mathrm{Aut}}
\newcommand{\Iso}{\mathrm{Iso}}
\newcommand{\Spec}{\mathrm{Spec}}

\newcommand{\MMN}{\mathcal{MM}_{\rm Nori}}

\newcommand{\Ah}{\mathcal{A}}
\newcommand{\Ch}{\mathcal{C}}
\newcommand{\isom}{\cong}
\newcommand{\id}{\mathrm{id}}
\newcommand{\colim}{\mathrm{colim}}
\newcommand{\ohne}{\smallsetminus}
\newcommand{\tensor}{\otimes}
\newcommand{\cone}{\mathrm{Cone}}
\newcommand{\tot}{\mathrm{Tot}}

\newcommand{\mot}{\mathrm{mot}}

\newcommand{\eff}{\mathrm{eff}}
\newcommand{\Rmod}{\text{$R$-Mod}}
\newcommand{\Qmod}{\text{$\Q$-Mod}}

\newcommand{\Rfree}{\text{$R$-Proj}}
\newcommand{\Nori}{\mathrm{Nori}}

\newcommand{\ZVar}{\Z[\mathrm{Var}]}
\newcommand{\ZAff}{\Z[\mathrm{Aff}]}
\newcommand{\dR}{\mathrm{dR}}
\newcommand{\ev}{\mathrm{ev}}

\newcommand{\Q}{\mathbb{Q}}
\newcommand{\Z}{\mathbb{Z}}

\newcommand{\C}{\mathbb{C}}
\newcommand{\G}{\mathbb{G}}
\renewcommand{\P}{\mathbb{P}}
\newcommand{\Pe}{\P}
\newcommand{\A}{\mathbb{A}}
\newcommand{\unit}{\mathbf{1}}
\newcommand{\Na}{\mathbb{N}}
\newcommand{\Ph}{\mathcal{P}}

\begin{document}
\title{On the relation between Nori Motives and Kontsevich Periods}
\author{Annette Huber and Stefan M\"uller-Stach}
\date{April 2, 2012}
\begin{abstract}We show that the spectrum of Kontsevich's algebra of formal periods is a torsor under
the motivic Galois group for mixed motives over $\Q$. 
This assertion is stated without proof by Kontsevich (\cite{Ko} Theorem 6) and
originally due to Nori. In a series of appendices, we also provide the necessary details on 
Nori's category of motives (see the survey \cite{LeNori}).
\end{abstract} 
\maketitle
\tableofcontents
\section*{Introduction}
Let $X$ be a smooth variety over $\Q$. A  {\em
period number} of $X$ is the value of an integral
\[ \int_Z\omega\]
where $\omega$ is a closed algebraic (hence rational) differential form of degree $d$ and
$Z\subset X(\C)$ a closed real submanifold of dimension $d$. Periods
are complex numbers. A more conceptual way is to view periods of $X$ as the
values of the period pairing between $H_d(X(\C),\Q)$ (singular homology)
and $H^d_\dR(X/\Q)$ (algebraic de Rham cohomology), or -- equivalently -- matrix entries 
of the period isomorphism \cite{D1,D2}
\[
H^d_\dR(X/\Q) \otimes \C \longrightarrow  H^d(X(\C),\Q)\otimes \C 
\]
in any two $\Q$--bases. Stated like this, the notion generalizes to any variety over $\Q$ or even any mixed motive over $\Q$, 
independently of the chosen category of mixed motives.

The algebra of Kontsevich-Zagier periods \cite{KZ}, defined as integrals of algebraic 
functions over domains described by algebraic equations or inequalities with coefficients in $\Q$ is the set of
all periods of all mixed motives over $\Q$. It is a very interesting,
countable subalgebra of $\C$. It contains e.g. all algebraic numbers, $2\pi i$, and
all $\zeta(n)$ for $n\in\Z$. Indeed, if Beilinson's conjectures are
true, then all special values of $L$-functions of mixed motives are periods.

The relations between period numbers are mysterious and intertwined with transcendence theory. A general 
period conjecture goes back to Grothendieck, see the footnote on page 358 in \cite{grothendieck}, and the reformulations 
of Andr\'e in \cite{andre1} and \cite{andre2} chapitre 23.
Conjecturally, the only relations are in a sense the obvious ones, i.e., coming from geometry. 
In order to make this statement precise, Kontsevich introduced (\cite{Ko} Definition 20) the notion of a formal period. We recall
his definition (actually, a variant, see Remark \ref{rem2.9}).

\begin{defn}
The space of {\em effective formal periods} $P^+$ is defined as the $\Q$-vector space generated
by symbols $(X,D,\omega,\gamma)$, where 
$X$ is an algebraic variety over $\Q$, $D\subset X$ a subvariety, 
$\omega\in H^d_{\mathrm{dR}}(X,D)$, $\gamma\in H_d(X(\C),D(\C),\Q)$ with 
relations
\begin{enumerate}
\item linearity in $\omega$ and $\gamma$;
\item for every $f:X\to X'$ with $f(D)\subset D'$
\[ (X,D,f^*\omega',\gamma)=(X',D',\omega',f_*\gamma)\]
\item for every triple $Z\subset Y\subset X$
\[ (Y,Z,\omega,\partial \gamma)= (X,Y, \delta  \omega,\gamma)\]
with $\partial$ the connecting morphism for relative singular homology and $\delta$ 
the connecting morphism for relative de Rham cohomology.
\end{enumerate}
$P^+$ is turned into an algebra via
\[ (X,D,\omega,\gamma) (X',D',\omega',\gamma')=(X\times X',D\times X'\cup X'\times D,\omega \wedge \omega', \gamma \times \gamma')\ .\]
The space of {\em formal periods} is the localization $P$ of $P^+$ with respect to
the period of $(\G_m,\{1\},\frac{dX}{X}, S^1)$ where $S^1$ is the unit circle in $\C^*$.
\end{defn}
The period conjecture then predicts the injectivity of the evaluation
map $P\to\C$. We have nothing to say about this deep conjecture, which
includes for example the transcendence of $\pi$ and all $\zeta(2n+1)$ for $n\in\Na$ \cite{andre1}.

The main aim of this note is to provide the proof (see Corollary \ref{cor3.4}) of the following
result:

\begin{thm}[Nori, \cite{Ko} Theorem 6]
$\Spec(P)$ is a torsor under the motivic Galois group of Nori's category
of mixed motives over $\Q$.
\end{thm}
As already explained by Kontsevich, singular cohomology and algebraic de Rham
cohomology are both fiber functors on the same Tannaka category of motives.
By general Tannaka formalism, there is a pro-algebraic torsor of isomorphisms between them. The period pairing
is nothing but a complex point of this torsor. Our task was to check
that this torsor is nothing but the explicit $\Spec(P)$. While baffling at
first, the statement turns out to be a corollary of the very construction
of Nori's motives.

This brings us to the second part of this paper. 
Nori's unconditional construction of an abelian category of mixed
motives has been around for some time. Notes of talks have been circulating,
see \cite{N} and \cite{N1}.
In his survey article \cite{LeNori}, Levine includes a sketch of the
construction, bringing it into the public domain. These sources combined contain basically all the necessary ideas.
We decided to work out all technical details that were not obvious to us, but were 
needed in order to get a full proof of our main theorem. This material
is contained in the appendices. They are supposed to be self-contained
and independent of each other. We tried to make clear where we are using Nori's ideas. All mistakes
remain of course ours.

Section \ref{sect1} is another survey 
on Nori motives, which brings the results of the appendices together in
order to establish that Nori motives are a neutral Tannakian category.
Section \ref{sect2} gives an interpretation of the algebra of formal
periods in terms of Nori's machine. The main result is Theorem \ref{thm2.6}.
In section \ref{sect3} this is combined with the rigidity property of Nori's
category of mixed motives in order to deduce the statement on the torsor
structure.

Appendix \ref{sectA} clarifies in detail the notion of torsor used by Kontsevich
in \cite{Ko} which was unfamiliar to us. 
Appendix \ref{sectB} is on the multiplicative structure on Nori's
diagram categories. Appendix \ref{sectC} establishes a criterion for
an abelian tensor category with faithful fiber functor to be rigid.
Both appendices are supposed to be applied to Nori motives but formulated
in general. 

The key geometric input in Nori's approach was the so-called basic lemma to which 
we refer here as Beilinson's lemma since A. Beilinson had proved a more general version earlier.
It allows for ''cellular'' decompositions of algebraic varieties, in the
sense that their cohomology looks like cohomology of a cellular
decomposition of a manifold. In appendix \ref{sectC} this is used
to compare different definitions of Nori motives using ''pairs''
or ''good pairs'' or even ''very good pairs''. The results of this
section are essential input in order to apply the rigidity criterion
of section \ref{yoga}.

What is missing from our account is the proof of universality of Nori's
diagram category for a diagram with a representation. The paper \cite{wangenheim} by J. von Wangenheim
provides full details. There is also 
the (unfortunately unpublished) paper \cite{brug} by A. Brughi\`eres
to fill in this point.

\noindent{\bf Acknowledgements:} We would like to thank Y. Andr\'e, F. Knop, M. Nori, W.~Soergel, D. van Straten,  M. Wendt, 
  and G. W\"ustholz for comments and discussions. 
We are particular thankful to  organizers and participants of the Z\"urich Mathematical School's 2011 workshop  ''Motives, periods and transcendence'' in Alpbach. We
thank them for their careful reading, comments and corrections of the first version of this paper. Special thanks go to
J. Ayoub, U. Choudhury, M. Gallauer, S.~Gorchinskiy, L. K\"uhne, and S. Rybakov.

\section{Essentials of Nori Motives}\label{sect1}
We use the setup of \cite{N} for Nori motives. The key ideas are contained in 
the survey \cite{LeNori}. Parts of the theory and further details are also developed in 
appendix \ref{diagrams}. Nori works primarily with singular homology. We have switched 
to singular cohomology throughout. We restrict to rational coefficients for simplicity.

We fix the following notation.
\begin{itemize}
\item By {\em variety} we mean a reduced separated scheme of finite type over $\Q$. 
\item Let $\Qmod$ be the category of finite dimensional $\Q$-vector spaces.
\item 
A \emph{diagram} $D$ is a directed graph.
\item  A \emph{representation} $T:D\to \Qmod$ assigns to every vertex in
$D$ an object in $\Qmod$ and to every edge $e$ from $v$ to $v'$ a homomorphism
$T(e):T(v)\to T(v')$.
\item Let $\Ch(D,T)$ its associated \emph{diagram category}. It is 
the universal abelian category together with a functor $f_T: \Ch(D,T) \to \Qmod$
such that $T$ factors via $f_T$. We often write $T$ instead of $f_T$. The category
$\Ch(D,T)$ arises as the category of $A(T)$-comodules finite dimensional over  $\Q$ for a certain coalgebra $A(T)$. 
 For the explicit construction due to Nori, 
see appendix \ref{diagrams} and \cite{wangenheim}. When $D$ or $T$ are understood from the context, we often abbreviate $\Ch(T)$ or $\Ch(D)$.
\end{itemize}
The following are the cases of interest in the present paper:

\begin{defn}\label{goodpairs} \label{pairs}
\begin{enumerate}
\item The diagram $D^\eff$ of {\em effective pairs} consists of triples
$(X,Y,i)$ with $X$ a $\Q$-variety, $Y\subset X$ a closed subvariety
and an integer $i$. 
There are two types of edges  between effective good pairs:
\begin{enumerate}
\item  (functoriality) For every morphism  
$f: X \to X'$ with $f(Y) \subset Y'$ an edge
\[ f^*:(X',Y',i) \to (X,Y,i)\ .\]
\item  (coboundary) For every chain $X \supset Y \supset Z$ of closed $\Q$-subschemes of $X$ an
edge
\[ \partial:(Y,Z,i) \to (X,Y,i+1)\ .\] 
\end{enumerate}
The diagram is graded (see Definition \ref{graded}) by $|(X,Y,i)|=i$.

\item
The diagram $D^\eff_\Nori$ of \emph{effective good pairs} is the full subdiagram of
$D^\eff$
with vertices the
triples $(X,Y,i)$ 
 such that singular cohomology satisfies
\[
H^j(X(\C),Y(\C);\Q)=0, \text{ unless } j=i.  
\]
\item 
The diagram $\tilde{D}^\eff$ of {\em effective very good pairs} is the full subdiagram of those
effective good pairs $(X,Y,i)$ with $X$ affine, $X\ohne Y$ smooth and either $X$ of dimension  $i$ and $Y$ of dimension $i-1$,
or $X=Y$ of dimension less than $i$. 
\end{enumerate}

The diagrams $D$ of {\em pairs}, 
$D_\Nori$ of \emph{good pairs} and $\tilde{D}$ of {\em very good pairs} are obtained 
by localization (see Definition \ref{localize}) with respect to $(\G_m,\{1\},1)$.
\end{defn}
We use the representation $H^*:D_\Nori\to \Qmod$ which assigns to
$(X,Y,i)$ relative singular cohomology $H^i(X(\C),Y(\C),\Q)$.

\begin{rem}For the purposes of our paper $\Q$-coefficients are sufficient.
Nori's machine also works for integral coefficients.
\end{rem}

Good pairs exist in abundance, see Appendix \ref{yoga}. 

\begin{defn}
The category of (effective) \emph{mixed Nori motives} $\MMN$ (resp. $\MMN^\eff$) is 
defined as the diagram category $\Ch(D_\Nori, H^*)$ (resp. $\Ch(D^\eff_\Nori,H^*)$). 
For a good pair $(X,Y,i)$ we write $H^i_\Nori(X,Y)$ for the corresponding object in $\MMN$.
We put
\[ 
\unit(-1)=H^1_\Nori(\G_m,\{1\})\in \MMN^\eff\ .
\]
\end{defn}

\addtocounter{lemma}{1}
\begin{rem}It will established in Corollary \ref{cor1.6} that Nori motives can equivalently be defined using $D$ or $\tilde{D}$.
\end{rem}
\begin{thm}
\begin{enumerate}
\item
 This definition is equivalent to Nori's original definition. 
\item
$\MMN^\eff \subset \MMN$ are commutative tensor categories with a faithful fiber functor $H^*$.
\item $\MMN$ is the localization of $\MMN^\eff$ with respect to the
Lefschetz object $\unit(-1)$.
\end{enumerate}
\end{thm}

\begin{proof} 
$D_\Nori^\eff$ is a graded diagram in the sense of Definition \ref{graded}.
It carries a commutative multiplicative structure (see Definition \ref{graded} again) by
\[ 
(X,Y,i)\times (X',Y',i')=(X\times X', X\times Y'\cup Y\times X',i+i').
\]
with unit given by $(\Spec \Q,\emptyset,0)$ and
\[ u:(X,Y,i)\to (\Spec \Q,\emptyset,0)\times (X,Y,i)=(\Spec \Q\times X,\Spec \Q\times Y,i)\]
be given by the natural isomorphism of varieties.
Let also
\begin{gather*}
 \alpha: (X,Y,i)\times (X',Y',i')\to (X',Y',i')\times (X,Y,i)\\
 \beta:(X,Y,i)\times\left( (X',Y',i')\times (X'',Y'',i'')\right)\to
          \left((X,Y,i)\times (X',Y',i')\right)\times (X'',Y'',i'')
\end{gather*}
be given by the natural isomorphisms of varieties.

$H^*$ is a graded representation in the sense of Definition \ref{graded}.
Properties (2) and (3) depend on a choice of a sign convention such that
the boundary map $\partial$ is compatible with cup products in the first 
variable and compatible up to sign in the second variable.

Hence by Proposition \ref{tensorcat}, the category $\MMN^\eff$ carries a tensor structure.

The Lefschetz object $(\G_m,\{1\},1)$ satisfies Assumption \ref{localassump}, 
hence by Proposition \ref{localcat} the category $\MMN$ is the localization
of $\MMN^\eff$ at $\unit(-1)$ and also a tensor category.

By definition, $D^\eff_\Nori$ is the category of cohomological good 
pairs in the terminology of \cite{LeNori}. In loc. cit. the category of
Nori motives is defined as the category of comodules of finite type over
$\Q$ for the localization of the ring $A^\eff$ with respect to the
element $\chi\in A(\unit(-1))$ considered in Proposition \ref{localcat}.
By this Proposition, the category of $A^\eff_\chi$-comodules agrees with
$\MMN$.
\end{proof}

\subsection*{Comparing diagrams and diagram categories}
Nori establishes that the representation $T=H^*$ extends to all pairs of varieties.

\begin{cor} \label{cor1.6} The diagram categories of $D^\eff$ and $\tilde{D}^\eff$ with
respect to singular cohomology are equivalent to $\MMN^\eff$ as abelian categories. The
diagram categories of $D$ and $\tilde{D}$ are equivalent to $\MMN$. 
\end{cor}
\begin{proof} In the following proof we omit the $T$ in the notation $\Ch(D,T)$. 
It suffices to consider the effective case.
The inclusion of diagrams induces faithful functors
\[ 
i:\Ch(\tilde{D}^\eff)\to \MMN^\eff\to \Ch(D^\eff).
\]
We are going to represent the diagram $D^\eff$ in $\Ch(\tilde{D}^\eff)$
such the restriction of the representation to $\tilde{D}^\eff$ gives
back $H^*$ (up to natural isomorphism).
By the universal property this induces a faithful functor
\[ j: \Ch(D^\eff)\to \Ch(\tilde{D}^\eff)\]
such that $j\circ i=\id$ (up to natural isomorphism). This implies
that $j$ is essentially surjective and full. Hence $j$ is an equivalence
of categories. This implies also that $i$ is an equivalence of categories.

We now turn to the construction of the representation of
$D^\eff$ in $\Ch(\tilde{D}^\eff)$.

We apply Proposition \ref{complexrepr}
to
\[ H^*:\tilde{D}^\eff\to \Ch(\tilde{D}^\eff)\]
and get a functor
\[ R: C_b(\ZVar)\to D^b(\Ch(\tilde{D}^\eff)).\]
Consider an  effective pair $(X,Y,i)$ in $D$. It is represented by
\[ H^i_\Nori(X,Y)=H^i(R(X,Y))\in \Ch(\tilde{D}^\eff)\]
where
\[ R(X,Y)=R(\cone(Y\to X))\ .\]
The construction
is functorial for morphisms of pairs. This allows to represent
edges of type $f^*$.

Finally, we need to consider edges corresponding to coboundary maps for triples $X\supset Y\supset Z$. 
In this case, it follows from the construction of $R$ that there is a natural triangle 
\[ R(X,Y)\to R(X,Z)\to R(Y,Z).\]
We use the connecting morphism in cohomology to represent the edge
$(Y,Z,i)\to (X,Y,i+1)$.
\end{proof}

\begin{cor}\label{cor1.7}Every object of $\MMN^\eff$ is subquotient of a direct sum of objects of
the form $H^i_\Nori(X,Y)$ for a good pair $(X,Y,i)$ where $X=W\ohne W_\infty$  and
$Y=W_0\ohne (W_0\cap W_\infty)$ with $W$ smooth projective, $W_\infty\cup W_0$ a divisor with normal crossings.
\end{cor}
\begin{proof} By Proposition \ref{criterion} every object in
the diagram category of $\tilde{D}^\eff$ (and hence $\MMN$) is subquotient of a direct sum of some
$H^i_\Nori(X,Y)$ with $(X,Y,i)$ very good. 

We follow Nori: By resolution of singularities there is a smooth projective variety $W$ and a normal crossing divisors
$W_0 \cup W_\infty \subset W$ together with a proper, surjective morphism $\pi: W \ohne W_\infty \to X$
such that one has $\pi^{-1}(Y)=W_0 \ohne W_\infty$ and $\pi: W \ohne \pi^{-1}(Y) \to X \ohne Y$ 
is an isomorphism. This implies that
\[ H^i_\Nori(W\ohne W_\infty, W_0\ohne (W_0 \cap W_\infty))\to H^i_\Nori(X,Y)\]
is also an isomorphism by proper base change, i.e., excision.
\end{proof}

\begin{rem} \label{rem1.9}Note that the pair $(W\ohne W_\infty,W_0\ohne (W_0 \cap W_\infty)$ is good, but not very good in general.
Replacing $Y$ by a larger closed subset $Z$, one may, however, assume that $W_0\ohne (W_0 \cap W_\infty)$ is affine.
Therefore, by Lemma~\ref{dualexist}, the dual of each generator can be assumed to be very good. 
 
It is not clear to us if it suffices to construct Nori's category using the diagram
of $(X,Y, i)$ with $X$ smooth, $Y$ a divisor with normal crossings.
The corollary says that the diagram category has the right ''generators'', but there might be too few ''relations''.
\end{rem}

\begin{cor}\label{les}
Let $Z\subset X$ be a closed immersion. Then there is
a natural object $H^i_Z(X)$ in $\MMN$ representing cohomology with
supports. There is a natural long exact sequence
\[\cdots \to H^i_Z(X)\to H^i_\Nori(X)\to H^i_\Nori(X\ohne Z)\to H^{i+1}_Z(X)\to\cdots\]
\end{cor}
\begin{proof} Let $U=X\ohne Z$. Put
\[ R_Z(X)=R(\cone(U\to X)),\hspace{3ex}H^i_Z(X)= H^i(R_Z(X))\ .\]
\end{proof}

\subsection*{Rigidity}

In order to establish duality, we need to check that Poincar\'e duality
is  motivic, at least in a weak sense.

\begin{defn}
Let $\unit(-1)=H^1_\Nori(\G_m)$ and $\unit(-n)=\unit(-1)^{\tensor n}$.
\end{defn}

\begin{lemma}\label{PD}
\begin{enumerate}
\item $H^{2n}_\Nori(\Pe^N)=\unit(-n)$ for $N \geq n \geq 0$.
\item Let $Z$ be a projective variety of dimension $n$. Then $H^{2n}_\Nori(Z)\isom \unit(-n)$.
\item Let $X$ be a smooth variety, $Z\subset X$ a smooth, irreducible, closed subvariety of pure codimension $n$. Then 
\[ H^{2n}_Z(X)\isom \unit(-n).
\]
\end{enumerate}
\end{lemma}
\begin{proof}
(1) Embedding projective spaces linearly into higher dimensional 
projective spaces induces isomorphisms on cohomology. Hence
it suffices to check the top cohomology of $\Pe^N$.

We start with $\Pe^1$. Consider the standard cover of $\Pe^1$ by
$U_1=\A^1$ and $U_2=\Pe^1\ohne\{0\}$. We have $U_1\cap U_2=\G_m$.
By Corollary \ref{cord.17} 
\[ R(\Pe^1)\to\cone\biggl(R(U_1)\oplus R(U_2)\to R(\G_m)\biggr)[-1]\]
is an isomorphism in the derived category.
This induces the isomorphism $H^2_\Nori(\Pe^1)\to H^1_\Nori(\G_m)$. 
Similarly, the \v{C}ech complex (see Definition \ref{chech}) for the standard affine cover of $\Pe^N$
relates $H^{2N}_\Nori(\Pe^N)$ with $H^N_\Nori(\G_m^N)$.

(2) Let $Z\subset \Pe^N$ be a closed immersion with $N$ large enough. 
Then $H^{2n}_\Nori(Z)\to H^{2n}_\Nori(\Pe^N)$ is an isomorphism in $\MMN$ because it
is in singular cohomology. 

(3)  
We note first that (3) holds in singular cohomology by the Gysin isomorphism 
\[
H^0(Z) {\buildrel \cong \over \longrightarrow} H^{2n}_Z(X) 
\]
under our assumptions. 
For the embedding $Z \subset X$ one has the deformation to the normal cone \cite[Sec. 5.1]{Fu}, i.e., a smooth scheme 
$D(X,Z)$ together with a morphism to $\A^1$ such that the fiber over $0$ is given by the normal bundle $N_Z X$ of $Z$ in $X$, and the other fibers by $X$. 
The product $Z \times \A^1$ can be embedded into $D(X,Z)$ as a closed subvariety of codimension $n$, inducing the embeddings of $Z \subset X$ as well as
the embedding of the zero section $Z \subset N_Z X$ over $0$.   
Hence, using the three Gysin isomorphisms and homotopy invariance, it follows that there are isomorphisms
\[
H^{2n}_Z(X) \leftarrow H^{2n}_{Z \times \A^1} (D(X,Z)) \rightarrow H^{2n}_Z(N_Z X)
\]
in singular cohomology and hence in our category.
Thus, we have reduced the problem to the embedding of the zero section $Z \hookrightarrow N_Z X$. 
However, the normal bundle $\pi: N_Z X \to Z$ trivializes on some dense open subset $U \subset Z$. This induces an isomorphism
\[
H^{2n}_Z(N_Z X) \to H^{2n}_U( \pi^{-1}(U)),
\]
and we may assume that the normal bundle $N_Z X$ is trivial. In this case, we have 
\[
N_Z(X)=N_{Z \times \{0\}}(Z \times \A^n) = N_{\{0\}}(\A^n),
\]
so that we have reached the case of $Z=\{0\} \subset \A^n$. Using the K\"unneth formula with supports and induction on $n$, it suffices 
to consider $H^{2}_{\{0\}}(\A^1)$ which is isomorphic to $H^1(\G_m)=\unit(-1)$ by Cor. \ref{les}. 
%
%
\end{proof} 

The following lemma (more precisely, its dual) is formulated implicitly in \cite{N} in order to
establish rigidity of $\MMN$.

\begin{lemma}\label{dualexist}
Let $W$ be a smooth projective variety of dimension $i$, $W_0,W_\infty\subset W$ divisors
such that $W_0\cup W_\infty$ is a normal crossings divisor.
Let 
\begin{align*}
   X&=W\ohne W_\infty\\
  Y&=W_0\ohne W_0\cap W_\infty\\
  X'&=W\ohne W_0\\
  Y'&=W_\infty\ohne W_0\cap W_\infty
\end{align*}
We assume
that $(X,Y)$ is a very good pair.

Then there is a morphism in $\MMN$
\[ q:\unit\to H^i_\Nori(X,Y)\tensor H^i_\Nori(X',Y')(i)
\]
such that the dual of $H^*(q)$ is a perfect pairing.
\end{lemma}
\begin{proof}
We follow Nori's construction. The two pairs are Poincar\'e dual to each
other in singular cohomology. (This is easily seen
by computing with sheaves on $W$ and the duality between $j_*$ and $j_!$).
This implies that they are both good pairs.
Hence
\begin{gather*}
H^i_\Nori(X,Y)\tensor 
H^i_\Nori(X',Y')
\to H^{2i}_\Nori( X\times X', X\times Y'\cup Y\times X')
\end{gather*}
is an isomorphism.
Let $\Delta=\Delta( W\ohne (W_0\cup W_\infty))$ via the diagonal map.
Note that
\[ X\times Y'\cup X'\times Y\subset X\times X'\ohne \Delta. \]
Hence by functoriality and the definition of cohomology with support, there is a map
\[
H^{2i}_\Nori( X\times X', X\times Y'\cup Y\times X')\leftarrow
H^{2i}_{\Delta}( X\times X').
\]
Again, by functoriality, there is a map
\[  H^{2i}_{\Delta}( X\times X')\leftarrow  H^{2i}_{\bar{\Delta}}(W\times W)\]
with $\bar{\Delta}=\Delta(W)$. By Lemma \ref{PD} it is isomorphic
to $\unit(-i)$. The map
$q$ is defined by twisting the composition by $(i)$. 
The dual of this map realizes Poincar\'e duality, hence it is a perfect
pairing.
\end{proof}

\begin{thm}[Nori] 
$\MMN$ is rigid, hence a neutral Tannakian category. Its Tannaka dual is given by $G_{\rm mot}=\Spec(A( D_\Nori,H^*))$. 
\end{thm}
\begin{proof}
By  Corollary \ref{cor1.7}
every object of
$\MMN$ is subquotient of  $M=H^i_\Nori(X,Y)(j)$ for a good pair $(X,Y,i)$
of the particular form occurring in Lemma \ref{dualexist}. By this Lemma they all admit a perfect pairing. 

By Proposition \ref{rigid}, the category $\MMN$ is neutral Tannakian. The Hopf algebra of its 
Tannaka dual agrees with Nori's algebra by Theorem \ref{tann}. 
\end{proof}


\section{Main Theorem}\label{sect2}

We describe the strategy of proof for Theorem 6 in \cite{Ko}:
The period algebra is given by the comparison of Nori motives
with respect to singular and de Rham cohomology.
The argument seems to be formal, so we do it abstractly.
Let $D$ be a graded diagram with commutative product structure (see Definition \ref{graded}), $T_1,T_2:D\to \Qmod$ two representations.

\begin{defn}
Let $A_1=A(D,T_1)$, $A_2=A(D,T_2)$. Put
\[ 
A_{1,2}=\colim_F \Hom(T_1|_F,T_2|_F)^\vee
\]
where $^\vee$ denotes the $\Q$-dual and $F$ runs through all
finite subdiagrams of $D$. 
\end{defn}

\begin{lemma}
$A_{1,2}$ is a commutative ring with multiplication induced by
the tensor structure of the diagram category.
The operation
\[ 
\End(T_1|_F)\times \Hom(T_1|_F,T_2|_F)\to \Hom(T_1|_F,T_2|_F)
\]
induces a compatible comultiplication
\[ 
A_1\tensor A_{1,2}\leftarrow A_{1,2}.
\]
\end{lemma}
\begin{proof}
The hard part is the existence of the multiplication. 
This follows by going through the proof of Proposition \ref{tensorcat},
replacing $\End(T|_{F})$ by $\Hom(T_1|_F,T_2|_F)$ in the appropriate places.
\end{proof}

\begin{ex} For $D_\Nori$, $T_1=H^*_{\dR}$ (de Rham cohomology) and $T_2=H^*$ (singular cohomology) 
this is going to induce the operation of the motivic Galois group $G_{\rm mot}$ on the torsor $X=\Spec A_{1,2}$.
\end{ex}

\begin{defn}\label{periodsabstract} We define the space of \emph{periods} $P_{1,2}$ as the $\Q$-vector
space generated by symbols
\[ 
(p,\omega,\gamma)
\]
where $p$ is a vertex of $D$, $\omega\in T_1(p)$, $\gamma\in T_2(p)^\vee$ with the following relations:
\begin{enumerate}
\item linearity in $\omega$, $\gamma$;
\item (change of variables) If $f:p\to p'$ is an edge in $D$, $\gamma\in T_2(p')^\vee$, $\omega\in T_1(p)$, then
\[ 
(p',T_1(f) (\omega), \gamma)=(p,\omega,T_2(f)^\vee(\gamma)).
\]
\end{enumerate}
\end{defn}

\begin{prop} $P_{1,2}$ is a commutative $\Q$-algebra with multiplication
given on generators by 
\[ (p,\omega,\gamma)(p',\omega',\gamma')=(p\times p',\omega \otimes \omega',\gamma\otimes\gamma')\]
\end{prop}

\begin{proof}
 It is obvious that the relations of $P_{1,2}$ are respected by the formula.
\end{proof}

There is a natural transformation
\[ \Psi:P_{1,2} \to A_{1,2}\]
defined as follows: let $(p,\omega,\gamma)\in P_{1,2}$. Let $F$ be
a finite diagram containing $p$.
Then 
\[ \Psi(p,\omega,\gamma)\in A_{1,2}(F)=\Hom(T_1|_F,T_2|_F)^\vee,\]
is the map
\[ \Hom(T_1|_F,T_2|_F)\to \Q\]
which maps $\phi\in \Hom(T_1|_F,T_2|_F)$ to $\gamma(\phi(p)(\omega))$.
Clearly this is independent of $F$ and respects relations of $P_{1,2}$.

\begin{thm}\label{thm2.6}
The above map
\[ 
\Psi:P_{1,2}\to A_{1,2}
\]
is an isomorphism of $\Q$-algebras.
\end{thm}

\begin{proof}
For a finite subdiagram $F\subset D$ let $P_{1,2}(F)$ be the space
of periods. By definition
$P=\colim_F P(F)$. The statement is compatible with
these direct limits. Hence without loss of generality $D=F$ is finite.

By definition $P_{1,2}$ is the subspace of
\[ \prod_{p\in D} T_1(p)\otimes T_2(p)^\vee\]
of elements satisfying the relations induced by $D$. By definition
$A_{1,2}(T)$ is the subspace of
\[ \prod_{p\in D}\Hom(T_1(p),T_2(p))^\vee\]
of elements satisfying the relations induced by $D$. As all
$T_i(p)$ are finite dimensional, this is the same thing.

The compatibility with coproducts is easy to see.
\end{proof}

\begin{rem}This works for coefficients in Dedekind rings as long as 
the representations take values in projective modules of finite type. The theorem is also of
interest in the case $T_1=T_2$. It then gives an explicit
description of Nori's coalgebra by generators and relations.
\end{rem}

Recall that de Rham cohomology of a smooth algebraic variety is defined as hypercohomology 
of the complex of differential forms. It is possible to extend
the definition naturally not only to singular varieties, but to pairs of varieties.
A possible reference is \cite{HuDiss} Section 7.

\begin{defn}\label{kontperiods}
The space of {\em effective formal periods} $P^+$ is defined as the $\Q$-vector space generated
by symbols $(X,D,\omega,\gamma)$, where 
$X$ is an algebraic variety over $\Q$, $D\subset X$ a subvariety, 
$\omega\in H^d_{\mathrm{dR}}(X,D)$, $\gamma\in H_d(X(\C),D(\C),\Q)$ with 
relations
\begin{enumerate}
\item linearity in $\omega$ and $\gamma$;
\item for every $f:X\to X'$ with $f(D)\subset D'$
\[ (X,D,f^*\omega',\gamma)=(X',D',\omega',f_*\gamma)\]
\item for every triple $Z\subset Y\subset X$
\[ (Y,Z,\omega, \partial \gamma)= (X,Y, \delta \omega,\gamma)\]
with $\partial$ the connecting morphism for relative singular homology and $\delta$ 
the connecting morphism for relative de Rham cohomology.
\end{enumerate}
$P^+$ is turned into an algebra via
\[ (X,D,\omega,\gamma) (X',D',\omega',\gamma')=(X\times X',D\times X'\cup D'\times X,\omega\wedge\omega', \gamma \times \gamma')\]
The space of {\em formal periods} is the localization $P$ of $P^+$ with respect to
the period of $(\G_m,\{1\},\frac{dX}{X}, S^1)$ where $S^1$ is the unit circle in $\C^*$.
\end{defn}

\begin{rem}\label{rem2.9}This is modeled after
Kontsevich \cite{Ko} Definition 20 but does not agree with it.
He restricts to smooth $X$ and $D$ a divisor with normal crossings.

The above definition uses effective pairs $(X,D,d)$ in the sense of Definition 
\ref{goodpairs}. By Corollary \ref{cor1.6},
it is clear that it suffices to take good or even very good pairs. 
By Corollary \ref{cor1.7}, it then suffices even to take generators
of Kontsevich's form.
However, we do not know, if these pairs generate all relations. Also, he only imposes relation (3) in a special case.  

Moreover, Kontsevich considers differential forms of top degree rather
than cohomology classes. This change is harmless. They are automatically closed.
He imposes Stokes's formula as an additional
relation, hence this amounts to considering cohomology classes.
\end{rem}

\begin{thm}\label{maintheorem}
Let $D$ be the diagram of pairs (see Definition \ref{goodpairs}).
Let
\begin{align*}
&T_1=H^*_\dR:D\to \Qmod\hspace{3ex} &\text{(de Rham cohomology)} \\
&T_2=H^*:D\to\Qmod \hspace{3ex} &\text{(singular cohomology)}
\end{align*}
then the space of formal periods $P$ (Definition \ref{periodsabstract})
agrees with the comparison algebra $A_{1,2}$ (Definition \ref{kontperiods}).
\[
P=P_{1,2}=A_{1,2}\ . 
\]
\end{thm}

\begin{proof}
We first restrict to the effective situation.
Let $P^\eff_{1,2}$ and $A^\eff_{1,2}$ be the algebras for this diagram.
Note that $H_d(X(\C),D(\C),\Q)$ is dual to $H^d(X(\C),D(\C),\Q)$. Hence
by definition $P^+=P^\eff_{1,2}$. 

Recall that there is a natural comparison isomorphism between
singular cohomology and de Rham cohomology (see e.g. \cite{HuDiss} \S 8) over
$\C$. Hence
by Theorem \ref{thm2.6} $P^\eff_{1,2}=A^\eff_{1,2}$ and $P_{1,2}=A_{1,2}$.
By localization and the analogue of Proposition \ref{localcat}, this
implies $P=A_{1,2}$.
\end{proof}

\begin{cor}\label{cor3.4} 
The algebra of formal periods $P$ remains unchanged when we restrict
in Definition \ref{kontperiods} to $(X,D,\omega,\gamma)$ with
$X$ affine of dimension $d$, $D$ of dimension $d-1$ and 
$X\ohne D$ smooth, $\omega\in H^d_\dR(X,D)$, $\gamma\in H_d(X(\C),D(\C),\Q)$.
\end{cor}
\begin{proof}
By the same proof, Theorem \ref{maintheorem} holds also for the diagram $\tilde{D}$ of very good pairs.
By the analogue of Corollary \ref{cor1.6}, the comparison 
algebra $A_{1,2}$ is the same for both diagrams. 
\end{proof}

All formal effective periods $(X,D,\omega,\gamma)$ can
be evaluated by ''integrating'' $\omega$ along $\gamma$. More precisely,
there is a natural pairing
\[ H^d_\dR(X,D)\times H_d(X(\C),D(\C))\to \C\]
This induces a ring homomorphism
\[ \ev: P\to \C\]
which maps $(\G_m,\{1\},dX/X,S^1)$ to $2\pi i$.
Numbers in the image of $\ev$ are called {\em Kontsevich-Zagier periods}.

\begin{cor}The algebra of Kontsevich-Zagier periods is generated by
$(2\pi i)^{-1}$ together with
periods of  $(X,D,\omega,\gamma)$ with $X$ smooth affine, $D$
a divisor with normal crossings, $\omega\in \Omega^d(X)$.
\end{cor}
\begin{proof}
Note that the period $2\pi i$ is of this shape.

By Corollary \ref{cor1.7} the category $\MMN^\eff$ is generated by
motives of good pairs $(X,Y,d)$ of the form 
$X=W\setminus W_\infty$, $Y=W_0\setminus (W_\infty\cap W_0)$ with
$W$ smooth projective of dimension $d$ and $W_0\cup W_\infty$ a divisor
with normal crossings. Hence their periods together with $(2\pi i)^{-1}$
generate the algebra of Kontsevich-Zagier periods.

By Remark \ref{rem1.9} we can assume that 
$X'=W\setminus W_0$ is affine. Let $Y'=W_\infty\setminus (W_0\cap W_\infty)$.
By Lemma \ref{dualexist} the motive
$H^d_\Nori(X,Y)$ is dual to 
$H^d_\Nori(X',Y')(d)$. Hence
the periods of $H^d_\Nori(X,Y)$ are in the algebra generated
by $(2\pi i)^{-1}$ and the periods of $H^d_\Nori(X',Y')$. 
As $X'$ is affine and $Y'$ a divisor with normal crossings, 
$H^d_\dR(X',Y')$ is generated by $\Omega(X')$.
\end{proof}


\section{Torsor structure on $P_{1,2}$}\label{sect3}
We return to the abstract setting. Let $D$ be a graded diagram with commutative product structure,
$T_1,T_2\to\Qmod$ two representations which become isomorphic after some
field extension $K/\Q$. The isomorphism is denoted by
\[
\varphi: T_1 \otimes K \to T_2 \otimes K. 
\]
Using the universal property of $\Ch(T_1)$ we immediately obtain:

\begin{lemma} $T_2$ extends to a fiber functor
\[
T_2: \Ch(T_1) \to \Qmod. 
\]
\end{lemma}

\begin{proof} 
Consider the abelian category $\Ah$ whose objects are pairs $(V_1,V_2)$ of
$\Q$-vector spaces together with an isomorphism $v:V_1\tensor K\to V_2\tensor K$. 
The data $T_1$, $T_2$ and $\varphi$ together define a representation 
\[ 
T:D\to \Ah\ .
\] 
Via the projection to the first component $T$ is compatible with $T_1$. 
The universal property (appendix B.10) implies that we have a commutative diagram of functors: 
\[
\begin{xy}\xymatrix{
&   \Ch(T_1)  \ar[dd]_{T}  \ar[rd]^{f_{T_1}}  & \\
D \ar[rr]_{} \ar[ru]^{\tilde T_1} \ar[rd]_{T} &   & \Qmod  \\
&  \Ah  \ar[ru]_{p_1} & \\
}\end{xy}
\]
The composition of $\Ch(T_1)\to \Ah$ with the projection $p_2$ to the second
component is the extension $T_2$.
\end{proof} 

In particular we have two fiber functors $T_1,T_2: \Ch(T_1) \to \Qmod$.

\begin{lemma}
Assume $\Ch(T_1)$ is a neutral Tannakian category. In this way we obtain
two affine group schemes $G_1=\Aut^\otimes(T_1)$, $G_2=\Aut^\otimes(T_2)$ and an affine scheme $X=X_{1,2}=\Iso^\otimes(T_1,T_2)$.
If $A_1$, $A_2$ and $A_{1,2}$ denote the Hopf algebras defined above then we have
\[
G_1=\Spec(A_1), \; G_2=\Spec(A_2), \; \textrm{ and } X=\Spec(A_{1,2}). 
\]
\end{lemma}

\begin{proof}
This follows (almost verbatim) the Tannakian pattern, see \cite{LNM900}. 
\end{proof}

In a similar way we can define an affine scheme $X_{2,1}=\Iso^\otimes(T_2,T_1)$.
These schemes are related via natural morphisms 
\[
X_{1,2} \times X_{2,1}  \longrightarrow  G_{1}, \quad X_{2,1} \times X_{1,2}  \longrightarrow  G_{2},
\]
and 
\[
G_1 \times X_{1,2} \longrightarrow X_{1,2}, \quad X_{1,2} \times G_2 \longrightarrow X_{1,2}.
\]

\begin{thm}
There is a natural isomorphism of affine schemes
\[
\iota: X_{1,2} \longrightarrow  X_{2,1}
\]
given by $f \mapsto f^{-1}$. Furthermore, $X_{1,2}$ and $X_{2,1}$ carry the structure of affine torsors 
in the sense of appendix A.
\end{thm}

\begin{proof} The assertion about $\iota$ is clear. 
Using the natural maps above, one obtains a commutative diagram
\[
\begin{xy}\xymatrix{
&  X_{1,2} \times X_{1,2} \times X_{1,2}    \ar[d]^{id \times \iota \times id} & \\
&  X_{1,2} \times X_{2,1} \times X_{1,2}    \ar@{-->}[dd]^{} \ar[rd] \ar[ld] & \\
G_1 \times X_{1,2} \ar[rd] &   & X_{1,2} \times G_2 \ar[ld]\\
& X_{1,2} &  \\
}\end{xy}
\]
with, as the composition of the two vertical maps, an induced morphism in the category of affine schemes
\[
(\cdot,\cdot,\cdot): X_{1,2} \times X_{1,2} \times X_{1,2} \longrightarrow  X_{1,2} 
\]
fitting into the diagram. It satisfies the axioms of a affine torsor, as defined in appendix A.
The assertion about $X_{2,1}$ is proved in a similar way.
\end{proof}

\begin{cor}
The algebra of formal periods $P$ (see Definition \ref{kontperiods})  has a natural torsor structure under $G_{\rm mot}$. 
\end{cor}

\begin{proof}  By Theorems 2.6 and 2.9 we have $A_{1,2}=P=P_{1,2}$ The previous theorem defines a
map 
\[
A_{1,2} \longrightarrow A_{1,2} \otimes A_{1,2} \otimes A_{1,2},
\]
i.e., a natural map
\[
P \longrightarrow P \otimes P \otimes P. 
\]
Finally note that $G_1=G_\mot$ is the motivic fundamental group by definition.
\end{proof}

\begin{rem}
In terms of period matrices this is given by the formula in \cite{Ko}:
\[
P_{ij} \mapsto \sum_{k,\ell} P_{ik} \otimes P_{k\ell}^{-1} \otimes P_{\ell j}.
\] 
The torsors in this section are naturally topological torsors in the pro-fppf topology. 
\end{rem}

\appendix
\section{Torsors}\label{sectA}

Kontsevich uses the following definition of torsors in \cite{Ko}. 
This notion at least goes back to a paper of R. Baer \cite{baer} from 1929, see the footnote
on page 202 of loc. cit. where Baer explains how the notion of a torsor comes up in the context of earlier 
work of H. Pr\"ufer \cite{pruefer}. In yet another context, ternary operations satisfying these axioms 
are called associative Malcev operations, see \cite{johnstone} for a short account.

\begin{defn}[\cite{baer} p. 202, \cite{Ko} p. 61, \cite{Fr} Definition 7.2.1]
A {\em torsor} is a set $X$ together with a map
\[ (\cdot,\cdot,\cdot):X\times X\times X\to X\]
satisfying:
\begin{enumerate}
\item $(x,y,y)=(y,y,x)=x$ for all $x,y\in X$
\item $((x,y,z),u,v)=(x,(u,z,y),v)=(x,y,(z,u,v))$ for all $x,y,z,u,v\in X$.
\end{enumerate}
Morphisms are defined in the obvious way, i.e., maps $X \to X'$ of sets commuting with the torsor structure.
\end{defn}
\begin{lemma} Let $G$ be a group. Then $(g,h,k)=gh^{-1}k$ defines a torsor structure
on $G$. 
\end{lemma}

\begin{proof} This is a direct computation:
\begin{align*} 
(x,y,y)&=xy^{-1}y=x=yy^{-1}x=(y,y,x),\\
((x,y,z),u,v)&=(xy^{-1}z,u,v)=xy^{-1}zu^{-1}v=(x,y,zu^{-1}v)=(x,y,(z,u,v)),\\
(x,(u,z,y),v)&=(x,uz^{-1}y,v)=x(uz^{-1}y)^{-1}v)=xy^{-1}zu^{-1}v.
\end{align*}
\end{proof}

\begin{lemma}[\cite{baer} page 202]\label{firstgroup} Let $X$ be a torsor, $e\in X$ an element. Then $G_e:=X$
carries a group structure via 
\[ 
gh:=(g,e,h),\hspace{2ex} g^{-1}:=(e,g,e).
\]
Moreover, the torsor structure on $X$ is given by the formula $(g,h,k)=gh^{-1}k$ in $G_e$.
\end{lemma}
\begin{proof}
First we show associativity:
\[ 
(gh)k=(g,e,h)k=((g,e,h),e,k)=(g,e,(h,e,k))=g(h,e,k)=g(hk).
\]
$e$ becomes the neutral element:
\[ 
eg=(e,e,g)=g; ge=(g,e,e)=g.
\]
We also have to show that $g^{-1}$ is indeed the inverse element:
\[ 
gg^{-1}=g(e,g,e)=(g,e,(e,g,e))=((g,e,e),g,e)=(g,g,e)=e.
\]
Similarly one shows that $g^{-1}g=e$. One gets the torsor structure back, since
\begin{align*}
 gh^{-1}k& =g(e,h,e)k=(g,e,(e,h,e))k=((g,e,(e,h,e)),e,k)\\
         & =(g,(e,(e,h,e),e),k)=(g,((e,e,h),e,e),k)\\
	 &   =(g,(h,e,e),k)=(g,h,k).
\end{align*}
\end{proof}

\begin{prop}\label{group}
Let $\mu_l:X^2\times X^2\to X^2$ be given by
\[ 
\mu_l\left((a,b),(c,d)\right)=((a,b,c),d).
\]
Then $\mu_l$ is associative and has $(x,x)$ for $x\in X$ as left-neutral 
elements.
Let $G^l=X^2/\sim_l$ where $(a,b)\sim_l (a,b)(x,x)$ for all $x\in X$ is an equivalence relation. Then
$\mu_l$ is well-defined on $G^l$ and turns $G^l$ into a group. Moreover,
the torsor structure map factors via a simply transitive left $G^l$-operation on $X$ which is defined by
\[
(a,b)x:=(a,b,x). 
\]
Let $e\in X$. Then
\[ 
i_e: G_e\to G^l, \hspace{1cm}x\mapsto (x,e)
\]
is group isomorphism inverse to $(a,b) \mapsto (a,b,e)$. \\
In a similar way, using $\mu_r\left((a,b),(c,d)\right):=(a,(b,c,d))$ we obtain 
a group $G^r$ with analogous properties acting transitively on the right on $X$ and such that 
$\mu_r$ factors through the action $X \times G^ r \to X$. 
\end{prop}

\begin{proof}
First we check associativity of $\mu_l$:
\begin{align*}
(a,b)[(c,d)(e,f)]&=(a,b)( (c,d,e),f)=((a,b,(c,d,e)),f)=
               (( (a,b,c),d,e),f)\\
[(a,b)(c,d)](e,f)&=((a,b,c),d)(e,f)=(((a,b,c),d,e),f)\\
\end{align*}
$(x,x)$ is a left neutral element for every $x \in X$:
\[ 
(x,x)(a,b)=((x,x,a),b)=(a,b)
\]
We also need to check that $\sim_l$ is an equivalence relation: $\sim_l$ is reflexive, since one has 
$(a,b)=((a,b,b),b)=(a,b)(b,b)$ by the first torsor axiom and the definition of $\mu$. 
For symmetry, assume $(c,d)=(a,b)(x,x)$. Then 
\begin{align*}
(a,b)&=((a,b,b),b)=((a,b,(x,x,b)),b)=(((a,b,x),x,b),b) \\
&=((a,b,x),x)(b,b)=(a,b)(x,x)(b,b)=(c,d)(b,b)
\end{align*} 
again by the torsor axioms and the definition of $\mu_l$. For transitivity observe that
\[
(a,b)(x,x)(y,y)=(a,b)((x,x,y),y)=(a,b)(y,y).
\]
Now we show that $\mu_l$ is well-defined on $G^l$:
\[ 
[(a,b)(x,x)][(c,d)(y,y)]=(a,b)[(x,x)(c,d)](y,y)=(a,b)(c,d)(y,y).
\]
The inverse element to $(a,b)$ in $G^l$ is given by $(b,a)$, since 
\[ 
(a,b)(b,a)=((a,b,b),a)=(a,a).
\]
Define the left $G^l$-operation on $X$ by $(a,b)x:=(a,b,x)$. 
This is compatible with $\mu_l$, since  
\begin{align*}
[(a,b)(c,d)]x&=((a,b,c),d)x=((a,b,c),d,x),\\
(a,b)[(c,d)x]&=(a,b)(c,d,x)=((a,b,(c,d,x))
\end{align*}
are equal by the second torsor axiom. 
The left $G^l$-operation is well-defined with respect to $\sim_l$:
\[ 
[(a,b)(x,x)]y=((a,b,x),x)y=((a,b,x),x,y)=(a,(x,x,b),y)=(a,b,y)=(a,b)y. 
\]
Now we show that $i_e$ is a group homomorphism:
\[
ab=(a,e,b) \mapsto ((a,e,b),e)=(a,e)(b,e)
\]
The inverse group homomorphism is given by 
\[ 
(a,b)(c,d)=((a,b,c),d)\mapsto((a,b,c),d,e).
\]
On the other hand in $G_e$ one has:
\[ 
(a,b,e)(c,d,e)=((a,b,e),e,(c,d,e))=(a,b,(e,e,(c,d,e)))=(a,b,(c,d,e)).
\]
This shows that $i_e$ is an isomorphism. The fact that $G_e$ is 
a group implies that the operation of $G^l$ on $X$ is simply transitive.
Indeed the group structure on $G_e=X$ is the one induced by the operation of $G^l$.
The analogous group $G^r$ is constructed using $\mu_r$ and an equivalence relation $\sim_r$ 
with opposite order, i.e., $(a,b)\sim_r (x,x)(a,b)$ for all $x\in X$. The properties of $G^r$
can be verified in the same way as for $G^l$ and are left to the reader.
\end{proof}

\begin{defn}
A {\em torsor} in the category of schemes is a scheme $X$ and a morphism
\[ 
X\times X\times X\to X
\]
which on $S$-valued points is a torsor for all $S$.
\end{defn}
This simply means that the diagrams of the previous definition commute
as morphisms of schemes. The following is the scheme theoretic version of Lemma \ref{group}.

\begin{prop}
Let $X$ be a torsor in the category of affine schemes. Then there are affine group schemes
$G^l$ and $G^r$ operating from the left and right on $X$ respectively such
that $X$ is a $G^l$- and $G^r$-torsor.
\end{prop}

\begin{proof}
We use Thm. 1.4 from \cite[expos\'e VII]{SGA3} to obtain affine scheme quotients by equivalence relations. 
In the case of $G^l$ we have to construct the quotient $G^l=X^2/\sim_l$. The case of $G^r$ is similar.
Then it follows also that the action homomorphism $G^l \times X \to X$ will be algebraic as the morphism $\mu_l$ 
also descends. In order to apply \cite{SGA3} we need to construct two morphisms 
$$
p_0,p_1: R \to X^2
$$
from an affine scheme $R$ to $X^2$, such that 
$$
R \to X^2 \times X^2
$$
is a closed embedding, for all $T$ the image $R(T) \to (X^2 \times X^2)(T)$ is an equivalence relation, 
i.e., reflexive, symmetric and transitive and such that the first projection $R \to X^2$ is flat. We choose 
\[
R:=X^2 \times X
\]
and set 
\[
p_0(a,b,x):=(a,b), \quad p_1(a,b,x):=(a,b)(x,x)=((a,b,x),x).
\]
Then the first projection $R \to X^2$ is flat. The image of $R$ is an equivalence relation by
Proposition \ref{group}. It remains to show that $(p_0,p_1): R \to X^2 \times X^2$ is a closed embedding. But this follows as 
$(p_0,p_1)(a,b,x)=((a,b),((a,b,x),x))$ is a graph type morphism. 
\end{proof}



\section{Localization and multiplication in diagrams}\label{sectB}

\label{diagrams}
A sketch of Nori's construction of motives is contained
in Levine's survey \cite{LeNori} \S 5.3.
We use this as a starting point
and develop the theory further to the extent needed for the proof of
the main theorem \ref{maintheorem} on periods. Full proofs for the
basics of diagram categories can be found in von Wangenheim's diploma thesis \cite{wangenheim}.  
\subsection{Diagrams and diagram categories}

Let $R$ be a noetherian ring. 

\begin{defn}
A \emph{small diagram} $D$ is a directed graph on a set of vertices 
such that for every vertex there is a distinguished edge
$\id:v\to v$. 
A diagram is called {\em finite} if it has only finitely many vertices. 
A {\em finite subdiagram} of a small diagram $D$ is a diagram containing a finite
subset of vertices of $D$ and all edges (in $D$) between them.
\end{defn}

\begin{rem}We added the notion of identity edges to Nori's definition given
in \cite{LeNori}.
It is useful when considering multiplicative structures.
\end{rem}

\begin{ex} Let $\Ch$ be a small category. 
Then we can associate a diagram 
$D(\Ch)$ with vertices the set of objects in $\Ch$ and edges given by morphisms.
\end{ex}

\begin{defn}[Nori]
A \emph{representation} $T$ of a diagram $D$ in a  category 
$\mathcal C$ is a map $T$ of directed graphs from $D$ to $D(\Ch)$ such that
$\id$ is mapped to $\id$. 
\end{defn}

We are particularly interested in categories of modules. 
\begin{defn} By $\Rmod$ we denote the category of
finitely generated $R$-modules. By $\Rfree$ we denote the subcategory of 
projective $R$-modules of finite type.
\end{defn}

Nori constructs a certain universal abelian category $\Ch(T)$ attached to 
a diagram and a representation $T$. 
For later use we recall Nori's construction.

\begin{defn}[Nori] 
Let $T$ be a representation of $D$ in $\Rmod$.
For each finite subdiagram $F\subset D$ let $\End(T|_F)$ be the ring
of endomorphisms of the functor $T|_F$, more precisely, as 
\[ 
\End(T|_F) := \left\{ (e_p)_{p \in F} \in \prod_{p\in F}\End_R(T(p)) \mid e_q \circ T(m)=T(m) \circ e_p \; \forall p,q \in F \; \forall m \in{\rm Mor}(p,q) \right\}.
\]

Let 
\[ \Ch(T|_F)=\End(T|_F)-\Mod\] be the category of finitely generated $R$-modules equipped with
an $R$-linear operation of algebra $\End(T|_F)$. Finally let
\[\Ch(T)=\colim_F \Ch(T|_F)\  .\]
\end{defn}

In the cases of most interest, there is more direct description.

\begin{prop}\label{propB.7}
Suppose $R$ is a Dedekind domain or a field and $T$ takes values in
$\Rfree$. Let $A(F,T):=\Hom_R(\End(T|_F),R)$.  We set
\[
A(T):=\colim_F A(F,T).
\]
Then $\Ch(T)$ is the category of finitely generated $R$-modules with an
$R$-linear $A(T)$-comodule structure.
\end{prop}
\begin{proof}
The assumptions on $R$ and $T$ ensure that $\End(T|_F,R)$ is a locally free $R$-module.
This allows to pass to the comodule description for finite diagrams, then
pass to the limit.

For full details see \cite{wangenheim} Satz 5.23. (He uses principal ideal domains. The arguments work without changes for Dedekind domains.) 
\end{proof}

The main step in Nori's construction of an abelian category of motives is the following result:

\begin{prop}[Nori's diagram category]\label{abelian_hull}
Let $R$ be a noetherian ring.
Let $D$ be a diagram and $T: D \to \Rmod$ be a representation. Then there
is a category $\Ch(T)$ together with a 
faithful exact $R$-linear functor $f_T: {\mathcal C}(T) \to  \Rmod$ 
and a representation $\tilde T: D \to {\mathcal C}(T)$
such that $f_T \circ \tilde T=T$ and $\tilde T$ is universal with respect to 
this property, i.e., for any other representation $F: D \to {\mathcal A}$
into some $R$-linear abelian category ${\mathcal A}$ with a faithful exact 
functor ${\mathcal A} \to \Rmod$ there is a unique functor (up to unique isomorphism)  
${\mathcal C}(T) \to {\mathcal A}$ such that
\[
\begin{xy}\xymatrix{
&   {\mathcal C}(T)  \ar@{-->}[dd]_{\exists !} \ar[rd]^{f_T}  & \\
D \ar[rr]_{T} \ar[ru]^{\tilde T} \ar[rd]^{F} &   & \Rmod \\
& {\mathcal A} \ar[ru] & \\
}\end{xy}
\]
commutes up to unique isomorphism.

$\Ch(T)$  is functorial in $D$ and $T$ in the obvious way.
\end{prop}

We are going to view $f_T$ as an extension of $T$ from $D$ to $\Ch(T)$ and
write simply $T$ instead of $f_T$.

\begin{proof}
\cite{LeNori, N}. A detailed proof is given in \cite{brug} or in \cite{wangenheim} Theorem 2.4. The condition on $R$ being noetherian is needed to ensure that
$\End(T|_F)$ is a finitely generated $R$-module.\end{proof}

The following properties hopefully allow a better understanding of the nature of
$\Ch(T)$.

\begin{prop}\label{criterion}
\begin{enumerate}
\item
As an abelian category $\Ch(T)$ is generated by the $\tilde{T}(v)$ where $v$ runs through the set of vertices of 
$D$, i.e., it agrees with its smallest full subcategory containing all such $\tilde{T}(v)$.  
\item Each object of $\Ch(T)$ is a subquotient of a finite direct sum of objects of the form $\tilde{T}(v)$.
\item If $\alpha: v \to v'$ is an edge in $D$ such that $T(\alpha)$ is an isomorphism, then 
$\tilde{T}(\alpha)$ is also an isomorphism. 
\end{enumerate}
\end{prop}
\begin{proof}
Let $\Ch' \subset \Ch(T)$ be the subcategory generated by all $\tilde{T}(v)$. By definition the representation $\tilde{T}$ factors through 
$\Ch'$. By the universal property of $\Ch(T)$ we obtain a functor $\Ch(T) \to \Ch'$, hence an equivalence of subcategories of $\Rmod$.

The second statement follows from the first criterion since the full subcategory in $\Ch(T)$ of subquotients of finite direct sums is abelian hence
agrees with $\Ch(T)$. 

The assertion on morphisms follows since the functor 
$f_T: \Ch(T) \to \Rmod$ is faithful and exact between abelian categories.
\end{proof}

\begin{thm}\label{tann}
Let $R$ be a field and
$\Ah$ be a neutral $R$-linear Tannakian category with fiber functor $T: D(\Ah) \to \Rmod$. Then $A(T)$ 
is equal to the Hopf algebra of the Tannakian dual. 
\end{thm}

\begin{proof}
By construction, see \cite{LNM900} Theorem 2.11 and its proof.
\end{proof}

We need to understand the behavior under base-change.
Let $S$ be a noetherian $R$-algebra. Then $T\otimes S$ is an $S$-representation of $D$.

\begin{lemma}[Base change]\label{basechange}
Let $S$ be a flat noetherian $R$-algebra and $T:D\to\Rfree$ a representation. Let $F\subset D$ be a finite subdiagram. Then: 
\begin{enumerate}
\item 
\[ \End_S(T\otimes S|_F)=\End_R(T|_F)\otimes S\]
\item If $R$ is a field or a Dedekind domain and $T$ takes values in $\Rfree$, then
\[ A(F,T\otimes S)=A(F,T)\otimes S\ ,\hspace{1cm}A(T\otimes S)=A(T)\otimes S\ .\]
\end{enumerate}
\end{lemma}
\begin{proof} We write $T_S=T\tensor S$.

$\End(T|_F)$ is defined as a limit, i.e., a kernel
\[ 
0\to \End(T|_F)\to \prod_{p\in O(D)}\End_R(T(p))\xrightarrow{\phi} \prod_{ m \in {\rm Mor}(p,q)} \Hom_R(T(p),T(q))
\]
with $\phi(p)(m)=e_q \circ T(m)-T(m) \circ e_p$. As $S$ is flat over $R$, this remains exact after $\otimes S$.
As $T(p)$ is projective, we have
\[ 
\End_R(T(p))\otimes S=\End_S(T\otimes S(p))
\]
Hence we get
\[ 
0\to \End(T|_F)\otimes S\to \prod_{p\in O(D)}\End_S(T_S(p))\xrightarrow{\phi} \prod_{m \in {\rm Mor}(p,q)}\Hom_S(T_S(p),T_S(q))
\]
This is the defining sequence for $\End(T_S|_F)$. This finishes the proof of the first statement.

Recall that $A(F,T)=\Hom_R(\End(T|_F),R)$. Both $R$ and $\End_R(T|_F)$ are projective because $R$ is now a field or a Dedekind domain. Hence
\[ \Hom_R(\End_R(T|_F),R)\tensor S\isom\Hom_S(\End_R(T|_F)\tensor S,S)\isom\Hom_S(\End_S((T_S)|_F),S)  .\]
This is nothing but $A(F,T_S)$.

Tensor products commute
with direct limit, hence the statement for $A(T)$ follows immediately.
\end{proof}

\subsection{Multiplicative structure}

Construction and properties of the tensor structure are
not worked out in detail in \cite{N}, \cite{N1} or \cite{LeNori}. In
particular, we were puzzled by the question how the
graded commutativity of the K\"unneth formula is dealt with 
in the construction. The following is an attempt to
clarify this on the formal level. The first version of this
section contained a serious mistake. We are particularly thankful to
Gallauer for pointing out both the mistake and the correction.

Recall that $\Rfree$ is the category of projective $R$-modules of finite
type for a fixed noetherian ring $R$.

\begin{defn}Let $D_1,D_2$ be small diagrams. Then
$D_1\times D_2$ is the small diagram with vertices
of the form $(f,g)$ for $f$ a vertex of $D_1$, $g$ a vertex of $D_2$, and with edges of the form
$(\alpha,\id)$ and $(\id,\beta)$ for $\alpha$ an edge of $D_1$ and $\beta$ an edge of $D_2$ and
with $\id=(\id,\id)$. 
\end{defn}

\begin{rem} Levine in \cite{LeNori} p.466 seems to define
$D_1\times D_2$ by taking the product of the graphs in the ordinary sense.
He claims  (in the notation of loc. cit.)
a map of diagrams 
\[ 
H_*\mathrm{Sch}'\times H_*\mathrm{Sch}'\to H_*\mathrm{Sch}'.
\]
We do not understand it on general pairs of edges. If $\alpha,\beta$ are
edges corresponding to boundary maps and hence lower the degree by $1$, then we would expect $\alpha\times \beta$ to lower the degree by $2$. However, there
are no such edges in $H_*\mathrm{Sch}'$.

Our restricted version of products of diagrams is enough to get
the implication.
\end{rem}

\begin{defn}\label{graded}
A {\em graded diagram} is a small diagram $D$ together with a map 
\[ 
|\cdot|: \{ \text{vertices of $D$}\ \}\to \Z/2\Z\ .\]
For an edge $\gamma:e\to e'$ we put $|\gamma|=|e|-|e'|$.
If $D$ is a graded diagram, $D\times D$ is equipped with the grading $|(f,g)|=|f|+|g|$.

A {\em commutative product structure} on a graded $D$ is a map of graded diagrams
\[ 
\times:D\times D\to D
\] 
together with choices of edges
\begin{align*}
\alpha_{f,g}:f\times g&\to g\times f\\
\beta_{f,g,h}:f\times(g\times h)&\to (f\times g)\times h\\
\beta'_{f,g,h}:(f\times g)\times h&\to f\times (g\times h)
\end{align*}
for all vertices $f,g,h$ of $D$.

A {\em graded multiplicative representation} $T$  of a graded diagram with commutative product
structure is a representation of $T$ in $\Rfree$ together with a choice of isomorphism  
\[
\tau_{(f,g)}: T(f\times g)\to   T(f)\otimes T(g)
\]
such that:
\begin{enumerate}
\item
The composition
\[ 
T(f)\otimes T(g)\xrightarrow{\tau_{(f,g)}^{-1}}T(f\times g)\xrightarrow{T(\alpha_{f,g})}T(g\times f)\xrightarrow{\tau_{(g,f)}}T(g)\otimes T(f)
\]
is $(-1)^{|f||g|}$ times the natural map of $R$-modules.
\item If $\gamma:f\to f'$ is an edge, then the diagram
\[\begin{CD}
T(f\times g)@>T(\gamma\times\id)>>T(f'\times g)\\
@V\tau VV @VV\tau V\\
T(f)\otimes T(g)@>(-1)^{|\gamma||g|}T(\gamma)\otimes \id>>T(f')\otimes T(g)
\end{CD}\]
commutes.
\item 
If $\gamma:f\to f'$ is an edge, then the diagram
\[\begin{CD}
T(g\times f)@>T(\id\times\gamma)>>T(g\times f')\\
@V\tau VV @VV\tau V\\
T(g)\otimes T(f)@>\id\otimes T(\gamma)>>T(g)\otimes T(f')
\end{CD}\]
commutes.
\item The diagram 
\[\begin{CD}
T(f\times (g\times h))@>T(\beta_{f,g,h})>> T((f\times g)\times h))\\
@VVV @VVV\\
T(f)\otimes T(g\times h)@.T(f\times g)\otimes T(h)\\
@VVV @VVV\\
T(f)\otimes ( T(g)\otimes T( h))@>>> (T(f)\otimes T(g))\otimes T(h)\\
\end{CD}\]
commutes under the standard identification
\[ 
T(f)\otimes\left( T(g)\otimes T(h)\right)\isom \left(T(f)\otimes T(g)\right)\otimes T(h).
\]
\item The maps $T(\beta_{f,g,h})$ and $T(\beta'_{f,g,h})$ are
inverse to each other.
\end{enumerate}
A {\em unit} for a graded diagram with commutative product structure $D$ is a vertex $\unit$ of degree $0$ together with a choice of edges
\[ u_f:f\to \unit\times f\]
for all vertices of $f$. A graded representation is {\em unital} if
$T(u_f)$ is an isomorphism for all vertices $f$.
\end{defn}
\begin{rem}
In particular, $T(\alpha_{f,g})$ and $T(\beta_{f,g,h})$ are isomorphisms. If $f=g$ then $T(\alpha_{f,f})=(-1)^{|f|}$.
If $\unit$ is a unit, then $T(\unit)$ satisfies $T(\unit)\isom T(\unit)\tensor T(\unit)$. Hence it is a free $R$-module of rank $1$.
\end{rem}

\begin{prop}\label{tensorcat}
Let $D$ be a graded diagram with commutative product structure with unit and
$T$ a unital graded representation of $D$ in $\Rfree$. 
\begin{enumerate}
\item
Then
$\Ch(T)$ is a commutative and associative tensor category with unit and $T:\Ch(T)\to\Rmod$ is a tensor functor. 
\item
If in addition $R$ is a field or a Dedekind domain, the coalgebra
$A(T)$
carries a natural structure of
commutative
bialgebra (with unit and counit).
\end{enumerate}
\end{prop}
The unit object is going to be denoted $\unit$.

\begin{proof}
We consider finite diagrams
$F$ and $F'$ such that
\[ 
\{ f\times g| f,g\in F\}\subset F'\ .
\]
We are going to define natural maps
\[ 
\mu_F^*:\End(T|_{F'})\to \End(T|_F)\otimes \End(T|_F).
\]
Assume this for a moment.
Let $X,Y\in \Ch(T)$. We want to define $X\otimes Y$ in $\Ch(T)=\colim_F\Ch(T|_F)$. Let $F$ such that
$X,Y\in\Ch(T|_F)$. This means that $X$ and $Y$ are finitely generated
$R$-modules with an action of $\End(T|_F)$. 
We equip the $R$-module $X\otimes Y$ with
a structure of $\End(T|_F')$-module. It is given by
\[ \End(T|_F')\tensor X\tensor Y\to \End(T|_F)\tensor \End(T|_F)\tensor X\tensor Y\to X\tensor Y\] 
where we have used the comultiplication map $\mu^*_F$ and the module structures of $X$ and $Y$.
This will be independent of the choice of $F$ and $F'$. Properties
of $\tensor$ on $\Ch(T)$ follow from properties of $\mu^*_F$.

If $R$ is a field or a Dedekind domain, let
\[\mu_F:A(F,T)\tensor A(F,T)\to A(F',T)\]
be dual to $\mu_F^*$. Passing to the direct limit defines a multiplication
$\mu$ on $A(T)$.

We now turn to the construction of $\mu_F^*$.
Let $a\in\End(T|_{F'})$, i.e., a compatible system of
endomorphisms $a_f\in \End(T(f))$ for $f\in F'$. We
describe its image $\mu^*_F(a)$. Let $(f,g)\in F\times F$. The isomorphism
\[ 
\tau:T(f\times g)\to T(f)\otimes T(g)
\]
induces an isomorphism
\[
\End (T(f\times g))\isom \End(T(f))\otimes \End(T(g)).
\]
We define the $(f,g)$-component of $\mu^*(a)$ by the image
of $a_{f\times g}$ under this isomorphism. 

In order to show that this is a well-defined element of 
$\End(T|_F)\otimes\End(T|_F)$, we need to check that diagrams of the form
\[
\begin{xy}\xymatrix{
T(f)\otimes T(g)\ar[r]^{\mu^*(a)_{(f,g)}}\ar[d]_{T(\alpha)\otimes T(\beta)}&T(f)\otimes T(g)\ar[d]^{T(\alpha)\otimes T(\beta)}\\
T(f')\otimes T(g')\ar[r]_{\mu^*(a)_{(f',g')}}&T(f')\otimes T(g')
}\end{xy}
\]
commute for all edges $\alpha:f\to f'$, $\beta:g\to g'$ in $F$.
We factor 
\[ 
T(\alpha)\otimes T(\beta)= (T(\id)\otimes T(\beta))\circ (T(\alpha)\circ T(\id))
\]
and check the factors separately.

Consider the diagram 
\[
\begin{xy}\xymatrix{
T(f\times g)\ar[rrr]_{a_{f\times g}}\ar[ddd]_{T(\alpha\times\id)}\ar[rd]^\tau&&&T(f\times g)\ar[ld]_\tau\ar[ddd]^{T(\alpha\times\id)}\\
&T(f)\otimes T(g)\ar[r]^{\mu^*(a)_{(f,g)}}\ar[d]_{T(\alpha)\otimes T(\id)}&T(f)\otimes T(g)\ar[d]^{T(\alpha)\otimes T(\id)}\\
&T(f')\otimes T(g)\ar[r]_{\mu^*(a)_{(f',g)}}&T(f')\otimes T(g)\\
T(f'\times g)\ar[rrr]^{a_{f'\times g}}\ar[ru]^\tau&&&T(f'\times g)\ar[ul]_\tau
}\end{xy}
\]
The outer square commutes because $a$ is a diagram endomorphism. Top and
bottom commute by definition of $\mu^*(a)$. Left and right commute
by property (3) up to the same sign  $(-1)^{|g||\alpha|}$.
Hence the middle square commutes without signs. 
The analogous diagram for  $\id\times\beta$ commutes on the nose.
Hence $\mu^*(a)$ is well-defined.

We now want to compare the $(f,g)$-component to the $(g,f)$-component.
Recall that there is a distinguished edge $\alpha_{f,g}:f\times g\to g\times f$. Consider the diagram
\[
\begin{xy}\xymatrix{
&T(f)\otimes T(g)\ar[r]^{\mu^*(a)_{(f,g)}} \ar[ddd]
& T(f)\otimes T(g) \ar[ddd]\\
T(f\times g)\ar[ru]^\tau\ar[d]_{T(\alpha_{f,g})}\ar[rrr]^{a_{f\times g}}&&&T(f\times g)\ar[lu]_\tau\ar[d]_{T(\alpha_{f,g})}\\
T(g\times f)\ar[rd]_\tau\ar[rrr]^{a_{f\times g}}&&&T(g\times f)\ar[ld]^\tau\\
&T(g)\otimes T(f)\ar[r]_{\mu^*(a)_{(g,f)}}&T(g)\otimes T(f)
} \end{xy}
\]
By the construction of $\mu^*(a)_{(f,g)}$ (resp. $\mu^*(a)_{(g,f)})$ the upper (resp. lower) tilted square
commutes. By naturality the middle rectangle with $\alpha_{f,g}$ commutes.
By property (1) of a representation of a graded diagram with
commutative product, the left and right faces commute
where the vertical maps are $(-1)^{|f||g|}$
times the natural commutativity of tensor products of $T$-modules. 
Hence the inner square also commutes without
the sign factors. This is cocommutativity of $\mu^*$.

The associativity assumption (3) for representations of diagrams with
product structure implies the coassociativity of $\mu^*$.

The compatibility of multiplication and comultiplication
is built into the definition.

In order to define a unit object in $\Ch(T)$ it suffices to define
a counit for $\End(T|_F)$. Assume $\unit\in F$. The counit
\[ u^*: \End(T|_F)\subset \prod_{f\in F}\End(T(f))\to \End(T(\unit))=R\]
is the natural projection. The assumption on unitality of $T$ allows
to check that the required diagrams commute. 
\end{proof}
\begin{rem}\label{weak_product}
The proof of Proposition \ref{tensorcat} works without any changes in the arguments when
we weaken the assumptions as follows: in Definition \ref{graded} replace
$\times$ by a map of diagrams
\[ \times:D\times D\to\Ph(D)\]
where $\Ph(D)$ is the path category of $D$: objects are the vertices of
$D$ and morphisms the paths. A representation $T$ of $D$ extends canonically
to a functor on $\Ph(D)$.
\end{rem}

\subsection{Localization}\label{ssectB.3}

The purpose of this section is to give a diagram version of the localization
of a tensor category with respect to one object, i.e., a distinguished object
$X$ becomes invertible with respect to tensor product. This is the standard
construction used to pass e.g. from effective Chow motives to all motives.
Again we thank Gallauer for pointing out a mistake in the original version as well as the correction.
 
We restrict to the case when $R$ is a field or a Dedekind domain and all representations of diagrams take values in $\Rfree$.

\begin{defn}[Localization of diagrams]\label{localize}

Let $D^\eff$ be a graded diagram with a commutative product structure with unit
$\unit$. Let $f_0\in D^\eff$ be a vertex. The \emph{localized diagram} $D$ has
vertices and edges as follows:
\begin{enumerate}
\item for every $f$ a vertex of $D^\eff$ and
$n\in\Z$ a vertex denoted $f(n)$;
 \item for every edge $\alpha:f\to g$ in $D^\eff$ and every $n\in\Z$, 
an edge denoted  $\alpha(n):f(n)\to g(n)$ in $D$;
\item for every vertex $f$ in $D^\eff$ and every $n\in\Z$ an edge 
denoted $\gamma_{f,n}:(f\times f_0)(n) \to f(n+1)$.
\end{enumerate}
 Put $|f(n)|=|f|$. 

We equip $D$ with a weak commutative product structure in the sense of Remark~\ref{weak_product} 
\[ \times:D\times D\to \Ph(D)\hspace{1cm}(f(n), g(m))\mapsto f\times g (n+m)\]
together with
\begin{align*} \alpha_{f(n),g(m)}&=\alpha_{f,g}(n+m)\\
              \beta_{f(n),g(m),h(r)}&=\beta_{f,g,h}(n+m+r)\\
	      \beta'_{f(n),g(m),h(r)}&=\beta'_{f,g,h}(n+m+r)
\end{align*}
Let $\unit(0)$ together with
\[ u_{f(n)}=u_f(n)\]
be the unit. 
\end{defn}
Note that there is a natural inclusion of multiplicative diagrams $D^\eff\to D$ which
maps a vertex $f$ to $f(0)$.

\begin{rem}
\begin{enumerate}
\item The above definition does not spell out $\times$ on edges.
It is induced from the product structure on $D^\eff$ for edges of type (2). For edges of type (3) there is an obvious sequence
of edges.
We take their composition in $\Ph(D)$. E.g. for $\gamma_{f,n}:(f\times f_0)(n)\to f(n+1)$ and $\id=\id(m):g(m)\to g(m)$ we have
\[ \gamma_{f,n}\times \id(m):(f\times f_0)(n)\times g(m)\to f(n+1)\times g(m)\]
via
\begin{multline*}
 (f\times f_0)(n)\times g(m)= ((f\times f_0)\times g)(n+m)\\
\xrightarrow{\beta'_{f,f_0,g}(n+m)} (f\times(f_0\times g))(n+m)\\
\xrightarrow{\id\times \alpha_{f_0,g}(n+m)} (f\times (g\times f_0))(n+m)\\
\xrightarrow{\beta_{f,g,f_0}(n+m)} ((f\times g)\times f_0)(n+m)\\
\xrightarrow{\gamma_{f\times g,n+m}} (f\times g)(n+m+1)
=f(n+1)\times g(m)\ .
\end{multline*}

\item In an earlier version, we had to assume $f_0$ even. We were
able to drop the condition by shifting the sign condition
in the definition of a multiplicative diagram to the first variable.
\end{enumerate}
\end{rem}

\begin{assump}\label{localassump}
Let $T$ be a graded multiplicative unital representation of $D^\eff$ with values in $\Rfree$ such that 
$T(f_0)$ is locally free of rank $1$ as $R$-module. 
\end{assump}

\begin{lemma}
$T$ extends uniquely to a graded multiplicative representation of $D$  
such that $T(f(n))=T(f) \otimes T(f_0)^{\otimes n}$ for all vertices
and $T(\alpha(n)=T(\alpha)\tensor T(\id)^{\tensor n}$ for all edges. 
It is multiplicative and unital with the choice
\[\begin{CD} 
T(f(n)\times g(m))@>\tau_{f(n),g(m)}>> T(f(n))\tensor T(g(m))\\
@V{\tau_{f,g}(n+m)}VV@VV=V\\
T(f)\tensor T(g)\tensor T(f_0)^{\tensor n+m}@>\isom >>
 T(f)\tensor T(f_0)^{\tensor n}\tensor T(g)\tensor T(f_0)^{\tensor m}
\end{CD}\]
where the last line is the natural isomorphism.
\end{lemma}

\begin{proof} Define $T$ on the vertices and edges of $D$ via the formula. 
It is tedious but straightforward to check the conditions.
\end{proof}

\begin{prop}\label{localcat}
Let $D^\eff,D$ and $T$ be as above. Let $A$ and $A^\eff$ be the corresponding bialgebras. Then:
\begin{enumerate}
\item
$\Ch(D,T)$ is
the localization of $\Ch(D^\eff,T)$ with respect to the object
$\tilde{T}(f_0)$.
\item  Let $\chi\in \End(T(f_0))^\vee=A(\{f_0\},T)$ be the dual of ${\rm id} \in \End(T(f_0))$. We view it in $A^\eff$.
Then $A=A^\eff_\chi$ (localization of algebras).
\end{enumerate}
\end{prop}

\begin{proof} Let $D^{\geq n}\subset D$ be the subdiagram with vertices
of the form $f(n')$ with $n'\geq n$. Clearly, $D=\colim_n D^{\geq n}$
and hence
\[ 
\Ch(D,T)\isom \colim_n\Ch(D^{\geq n},T)\ .
\]
Consider the morphism of diagrams
\[ 
D^{\geq n}\to D^{\geq n+1},\hspace{1ex} f(m)\mapsto f(m+1). 
\]
It is clearly an isomorphism. We equip $\Ch(D^{\geq n+1})$ with
a new fibre functor $f_T\tensor T(f_0)^\vee$. It is faithful exact.
The map $f(m)\mapsto \tilde{T}(f(m+1))$ is a representation of (Index) $D^{\geq n}$ in the abelian category
$\Ch(D^{\geq n+1},T)$ with fibre functor $f_T\tensor T(f_0)^\vee$. By the universal property this induces a functor
\[ \Ch(D^{\geq n},T)\to \Ch(D^{\geq n+1},T)\ .
\]
The converse functor is constructed in the same way. Hence 
\[ \Ch(D^{\geq n},T)\isom \Ch(D^{\geq n+1},T), \hspace{0.5cm}A^{\geq n}\isom A^{\geq n+1}.
\]
The map of graded diagrams with commutative product and unit
\[ D^\eff\to D^{\geq 0}\]
induces an equivalence on tensor categories. Indeed, we represent $D^{\geq 0}$
in $\Ch(D^\eff,T)$ by mapping $f(m)$ to $\tilde{T}(f)\otimes \tilde{T}(f_0)^m$.
By the universal property, this implies that there is a faithful exact
functor 
\[ 
\Ch(D^{\geq 0},T)\to \Ch(D^\eff,T)
\]
inverse to the obvious inclusion. Hence we also have $A^\eff\isom A^{\geq 0}$
as unital bialgebras.

On the level of coalgebras, this implies
\[ 
A=\colim A^{\geq n}=\colim A^\eff
\]
with $A^{\geq n}=A(D^{\geq n},T)$ isomorphic to $A^{\eff}$ as coalgebras.
$A^\eff$ also has a multiplication, but the $A^{\geq n}$ do not. However, they
carry a weak $A^\eff$-module structure analogous to Remark \ref{weak_product} corresponding to the map of graded  diagrams
\[ 
D^\eff\times D^{\geq n}\to \Ph(D^{\geq n}).
\]
We want to describe the transition maps of the direct limit. From the point
of view of $D^\eff\to D^\eff$ it is given by $f\mapsto f\times f_0$.

In order to describe $A^\eff\to A^\eff$ it suffices to
describe $\End(T|_F)\to \End(T|_{F'})$ where
$F,F'$ are finite subdiagrams of $D^\eff$ such that $f\times f_0\in O(F')$ for
all vertices $f\in O(F)$. It is induced by
\[ 
\End(T(f))\to \End(T(f\times f_0))\xrightarrow{\tau}\End(T(f))\otimes \End(T(f_0))\hspace{2ex} a\mapsto a\otimes \id.
\]
On the level of coalgebras this corresponds to the map
\[ 
A^\eff\to A^\eff, \hspace{2ex} x\mapsto x\chi.
\]
Note finally, that the direct limit $\colim A^\eff$ with transition maps
given by multiplication by $\chi$ agrees with the localization $A^\eff_\chi$.
\end{proof}

\section{Nori's Rigidity Criterion}\label{sectC}
Implicit in Nori's construction of motives is a rigidity criterion, which
we are now going to formulate and prove explicitly.

Let $R$ be a Dedekind domain or a field and $\Ch$ an $R$-linear tensor category.
Recall that $\Rmod$ is the category of finitely generated $R$-modules and
$\Rfree$ the category of finitely generated projective $R$-modules.

 We assume that
the tensor product is associative, commutative and unital. Let $\unit$ be the unit object. 
Let $T:\Ch\to \Rmod$ be a faithful tensor functor with values in $\Rmod$. In particular, $T(\unit)=R$.

We introduce an ad-hoc notion.
\begin{defn}Let $V$ be an object of $\Ch$. We say that $V$ {\em admits
a perfect duality} if there is morphism
\[ q:V\tensor V\to \unit\]
or
\[ \unit\to V\tensor V\]
such that $T(V)$ is projective and $T(q)$ (respectively its dual) is a non-degenerate bilinear form.
\end{defn} 

\begin{defn}Let $V$ be an object of $\Ch$. By $\langle V\rangle_\tensor$ we denote
the smallest full abelian unital tensor subcategory of $\Ch$ containing
$V$.
\end{defn}
We start with the simplest case of the criterion.
\begin{lemma}\label{single}Let $V$ be an object such that $\Ch=\langle V\rangle_\tensor$ and such
that $V$ admits a perfect duality. Then $\Ch$ is rigid.
\end{lemma}
\begin{proof} By standard Tannakian formalism, $\Ch$ is the category of
comodules for a bialgebra $A$, which is commutative and of finite type as an $R$-algebra. 
Indeed, the construction of $A$ as a coalgebra was explained in
Proposition \ref{propB.7}.
We want to show that $A$ is a Hopfalgebra, or
equivalently, that the algebraic monoid $M=\Spec A$ is an algebraic group.

By Lemma \ref{submonoid} it suffices to show 
that there is a closed immersion $M\to G$ of monoids into an
algebraic group $G$. We are going to construct this group or rather its
ring of regular functions. We have 
\[ A=\lim A_n\]
with $A_n$ the Tannakian dual of $\Ch_n=\langle\unit,V,V^{\tensor 2},\dots,V^{\tensor n}\rangle$,
the smallest full abelian subcategory containing $\unit,V,\dots,V^{\tensor n}$.
By construction there is a surjective map
\[ \bigoplus_{i=0}^n\End_R((T(V)^{\tensor i})^\vee\to A_n\]
or dually an injective map
\[ A_n^\vee\to \bigoplus_{i=0}^n\End_R(T(V)^{\tensor i})\]
where $A_n^\vee$ consists of those endomorphisms compatible with all
morphisms in $\Ch_n$. In the limit there is a surjection of bialgebras
\[ \bigoplus_{i=0}^\infty\End_R((T(V)^{\tensor i})^\vee)\to A\]
and the kernel is generated by the relation defined by compatibility
with morphisms in $\Ch$. 
One such relation is the commutativity constraint, hence the map factors via the symmetric algebra
\[ S^*( \End(T(V)^\vee)\to A\ .\]
Note that $S^*(\End(T(V)^\vee)$ is canonically the ring of regular functions
on the algebraic monoid $\End(T(V))$. Another morphism in $\Ch$ is the
pairing $q:V\tensor V\to\unit$. We want to work out the explicit equation
induced by $q$.

We choose a basis $e_1,\dots,e_r$ of $T(V)$. Let 
\[ a_{i,j}=T(q)(e_i,e_j)\in R\]
By assumption the matrix is invertible.
 Let $X_{st}$ be the matrix coefficients on
$\End(T(V))$ corresponding to  the basis $e_i$. Compatibility with $q$ gives
for every pair $(i,j)$ the equation
\begin{align*}
 a_{ij}&=q(e_i,e_j)\\
               &=q((X_{rs})e_i,(X_{r's'})e_j)\\
                &=q\left(\sum_{r}X_{ri}e_r,\sum_{r'} X_{r'j}e_{r'}\right)\\
                &=\sum_{r,r'}X_{ri}X_{r'j}q(e_{r},e_{r'})\\
                &=\sum_{r,r'} X_{ri}X_{r'j}a_{rr'}
\end{align*}
Note that the latter is the $(i,j)$-term in the product of matrices
\[ (X_{ir})^t(a_{rr'})(X_{r'j}) \ .\]
Let $(b_{ij})=(a_{ij})^{-1}$. 
With 
\[ (Y_{ij})=(b_{ij}) (X_{i'r})^t(a_{rr'})\]
we have the coordinates of the inverse matrix.
In other words, our set of equations defines the isometry group $G(q)\subset \End(T(V))$. We now have expressed
$A$ as quotient of the ring of regular functions of $G(q)$.

The argument works in the same way, if we are given
\[ q:\unit\to V\tensor V\]
instead.
\end{proof}

\begin{prop}[Nori]\label{rigid}Let $\Ch$ and $T:\Ch\to \Rmod$ be as defined at the beginning of the section. Let $\{V_i|i\in I\}$
be a set of objects of $\Ch$ with the properties:
\begin{enumerate}
\item It generates $\Ch$ as an abelian tensor category, i.e., the smallest
full abelian tensor subcategory of $\Ch$ containing all $V_i$ is equal to $\Ch$.
\item For every $V_i$ there is an object $W_i$ and  a morphism
 \[ q_i:V_i\tensor W_i\to \unit\]
such
that $T(q_i):T(V_i)\tensor T(W_i)\to T(\unit)=R$ 
 is a perfect pairing of free
$R$-modules.
\end{enumerate}
Then $\Ch$ is rigid, i.e., for every object $V$ there is a dual object $V^\vee$ such that
\[\Hom(V\otimes A,B)=\Hom(A,V^\vee\otimes B)\ ,\hspace{1cm}\Hom(V^\vee\tensor A,B)=\Hom(A,V\tensor B)\ .\]
\end{prop}
This means that the Tannakian dual of $\Ch$ is not only a monoid but a group.
\begin{rem}The Proposition also holds with the dual assumption, existence of
morphisms
\[ q_i:\unit\to V_i\tensor W_i\]
such that
$T(q_i)^\vee:T(V)^\vee\tensor T(W_i)^\vee\to R$
 is a perfect pairing.
\end{rem}

\begin{proof}
Consider $V'_i=V_i\oplus W_i$. The pairing $q_i$ extends to a symmetric map
$q'_i$ on $V'_i\otimes V'_i$ such that $T(q'_i)$ is non-degenerate. We now replace
$V_i$ by $V'_i$.
Without loss of generality, we can assume
$V_i=W_i$.

For any finite subset $J\subset I$ let $V_J=\bigoplus_{j\in J}V_j$. Let
$q_J$ be the orthogonal sum of the $q_j$ for $j\in J$. It is again a symmetric perfect
pairing. 

For every object $V$ of $\Ch$ we write $\langle V\rangle_\tensor$ for the smallest
full abelian tensor subcategory of $\Ch$ containing $V$. By assumption we
have 
\[ \Ch=\bigcup_J \langle V_J\rangle_\tensor\]
We apply the standard Tannakian machinery. It attaches to every $\langle V_J\rangle_\tensor$
an $R$-bialgebra $A_J$ such that $\langle V_J\rangle_\tensor$ is equivalent to
the category of $A_J$-comodules. If we put
\[ A=\lim A_J\]
then $\Ch$ will be equivalent to the category of $A$-comodules. It suffices
to show that $A_J$ is a Hopf-algebra. This is the case by Lemma \ref{single}.
\end{proof}

Finally, the missing lemma on monoids.

\begin{lemma}\label{submonoid}
Let $R$ be noetherian ring, $G$ be an algebraic group scheme of finite type over $R$ and $M\subset G$ a closed immersion
of a submonoid with $1\in M(R)$. Then $M$ is an algebraic group scheme over $R$.
\end{lemma}
\begin{proof}This seems to be well-known. It is appears as an exercise in
\cite{Renner} 3.5.1 2.
We give the argument: 

Let $S$ be any finitely generated $R$-algebra. We have to show that the functor $S \mapsto M(S)$ is a group. 
We take base change of the situation to $S$. Hence without loss of generality, it suffices to consider $R=S$. 
If $g \in G(R)$, we denote the isomorphism $G \to G$ induced by left multiplication with $g$ also by $g: G \to G$.
Take any $g \in G(R)$ such that $gM \subset M$ (for example $g \in M(R)$). Then one has 
\[
M \supseteq gM \supseteq g^2M \supseteq \cdots  
\]
As $G$ is Noetherian, this sequence stabilizes, say at $s \in \mathbb N$:
\[
g^sM=g^{s+1}M
\]
as closed subschemes of $G$. Since every $g^s$ is an isomorphism, we obtain that
\[
M=g^{-s} g^sM= g^{-s} g^{s+1}M = gM
\]
as closed subschemes of $G$. So for every $g \in M(R)$ we showed that $gM=M$. Since $1 \in M(R)$, 
this implies that $M(R)$ is a subgroup.
\end{proof}

\section{Yoga of good pairs}\label{yoga}

We recall the definition of good pairs from the main text.

Let $R$ be noetherian ring, $k$ a subfield of $\C$. A variety is
a separated reduced scheme of finite type over $k$. We denote by
$X(\C)$ the set of complex points equipped with the analytic topology.

We denote by $\ZVar$ the additive
category whose objects are varieties over $k$ and whose morphisms
on connected varieties are formal $\Z$-linear combinations of morphisms between varieties.
We denote $\ZAff$ the full subcategory whose objects are affine varieties.

By base change to $\C$ we
can consider the corresponding analytic space and its singular cohomology.
\begin{defn}\label{verygood}
\begin{enumerate}
\item A triple $(X,Y,i)$ of a variety $X$ and a closed subvariety $Y$ and
an integer $i$ is called
{\em good pair} if singular cohomology satisfies
\[
H^j(X(\C),Y(\C);R)=0, \text{ unless } j=i.  
\]
and $H^i(X(\C),Y(\C);R)$ is free.
\item The diagram $D^\eff$ of good pairs has as vertices good pairs.
There are two types of edges  between effective good pairs: first the edges induced by morphisms $f^*:(X',Y',i) \to (X,Y,i)$ 
of triples for $f: X \to X'$ and $f(Y) \subset Y'$. The second type of edges 
$\partial:(Y,Z,i) \to (X,Y,i+1)$ arises for every chain $X \supset Y \supset Z$ of closed $k$-subvarieties of $X$ (coboundary). 
\item A good pair is called {\em very good} if $X$ is affine and
$X\ohne Y$ smooth and either $X$ of dimension $i$ and $Y$ of dimension $i-1$ or $X=Y$ of dimension less than $i$.
\item The diagram $\tilde{D}^\eff$ of very good pairs has
as vertices the very good pairs and edges as in $D^\eff$. 
\end{enumerate}
\end{defn}

\begin{lemma}[Basic Lemma of Nori]\label{basiclemma} 
Let $X$ be an affine scheme of finite type over $k \subset \C$ of dimension $n$ and $Z \subset X$ be a closed subscheme of dimension $\le n-1$. 
Then there is a closed subscheme $Y \supset Z$ such that 
$(X,Y,n)$ is a good pair.
\begin{itemize}
\item $\dim(Y) \le n-1$.
\item $H^i(X(\C),Y(\C);R)=0$ for $i \neq n$.
\item $H^n(X(\C),Y(\C);R)$ is a free $R$-module.
\end{itemize}
Moreover, $X\ohne Y$ can be chosen smooth.
\end{lemma}

A similar result holds in arbitrary characteristic by work of Beilinson \cite[Lemma 3.3]{beilinson} and 
Kari Vilonen apparently used similar methods in his Master thesis \cite{vilonen}.

The aim of this appendix is to establish the following result.

\begin{prop}\label{complexrepr} Let $\Ah$ be an $R$-linear abelian category with a faithful 
forgetful functor to $\Rmod$. Let $T: \tilde{D}^\eff\to \Ah$ be a representation. 
Then there is a natural contravariant triangulated functor
\[ 
R:C_b(\ZVar)\to D^b(\Ah)
\]
on the category of bounded homological complexes in $\ZVar$
such that for every good pair $(X,Y,i)$ we have 
\[
H^j( R(\cone (Y\to X))=\begin{cases} 0&j\neq i\\
                               T(X,Y,i)&j=i
                               \end{cases}
\] 
Moreover, the image of $R(X)$ in $D^b(\Rmod)$ computes
singular cohomology of $X(\C)$.
\end{prop}

We are mostly interested in two explicit examples of complexes.
\begin{defn}\label{relative_and_support}
Consider the situation of Proposition \ref{complexrepr}.
Let $Y\subset X$ be a closed subvariety with open complement $U$, $i\in\Z$. Then we put
\begin{gather*} R(X,Y)=R(\cone(Y\to X)),\hspace{2ex}R_Y(X)=R(\cone(U\to X))\in D^b(\Ah)\\
T(X,Y,i)=H^i(R(X,Y)),\hspace{2ex} T_Y(X,i)=H^i(R_Y(X))\in \Ah
\end{gather*}
$T(X,Y,i)$ is called {\em relative cohomology}. $T_Y(X,i)$ is called {\em cohomology with support}. 
\end{defn}

The strategy of the proof combines a variation of constructions in \cite{HuberReal}
and a key idea of Nori.

The first step is to replace arbitrary complexes by affine ones.
The idea for the following construction is from the \'etale case,
see \cite{Friedlander} Definition 4.2.
\begin{defn} Let $X$ a variety. A {\em rigidified} affine cover is
a finite open affine covering $\{U_i\}_{i\in I}$ together with a choice of an index $i_x$ 
for every closed point $x\in X$ such that $x \in U_{i_x}$. We also assume that in the covering 
every index $i \in I$ occurs as $i_x$ for some $x \in X$. 

Let $f:X\to Y$ be a morphism of varieties, $\{ U_i\}_{i\in I}$ a
rigidified open cover of $X$ and $\{V_j\}_{j\in J}$ a rigidified
open cover of $Y$. 
A {\em morphism} of rigidified covers (over $f$) 
\[ \phi:\{ U_i\}_{i\in I}\to \{V_j\}_{j\in J}\]
is a map of sets $\phi:I\to J$ such that $f(U_i)\subset V_{\phi(i)}$ and
for all $x\in X$ we have $\phi(i_x)=j_{f(x)}$.
\end{defn}
\begin{rem} Under these conditions the rigidification makes $\phi$ unique if it exists.\end{rem}

\begin{lemma} The projective system of rigidified affine covers  is filtered
and strictly functorial, i.e.,
if $f:X\to Y$ is a morphism of varieties, pull-back defines
a map of projective systems.
\end{lemma}
\begin{proof}Any two covers have their intersection as common refinement.
The rigidification extends in the obvious way. Preimages
of rigidified covers are rigidified open covers.
\end{proof}

\begin{defn}
Let $F=\sum a_if_i:X\to Y$ be a morphism in $\ZVar$. The {\em support} of $F$
is the set of $f_i$ with $a_i\neq 0$. 

Let $X_*$ be a homological complex of varieties, i.e., an object in $C_b(\ZVar)$. 
An {\em affine cover} of $X_*$ is a complex of rigidified affine covers, i.e., for
every $X_n$ the choice of a rigidified open cover $\tilde{U}_{X_n}$ and for 
every $g:X_n\to X_{n-1}$ in the support of the differential $X_n\to X_{n-1}$ in the complex $X_*$ 
a morphism of rigidified covers $\tilde{g}:\tilde{U}_{X_n}\to \tilde{U}_{X_{n-1}}$
over $g$. 

Let $F_*:X_*\to Y_*$ be a morphism in $C_b(\ZVar)$ and $\tilde{U}_{X_*}$, $\tilde{U}_{Y_*}$ affine covers of $X_*$ and $Y_*$.  A morphism of affine
covers over $F_*$ is a morphism of ridigied affine covers $f_n:\tilde{U}_{X_n}\to \tilde{U}_{Y_n}$ over every morphism in the support of $F_n$.
\end{defn}

\begin{lemma}Let $X_*\in C_b(\ZVar)$. Then the projective system of 
rigidified affine covers of $X_*$ is non-empty, filtered and functorial, i.e.
if $f_*:X_*\to Y_*$ is a morphism of complexes and $\tilde{U}_{X_*}$ an affine
cover of  $X_*$, then there is affine cover $\tilde{U}_{Y_*}$ and a morphism or complexes of rigidified affine covers. Any two choices are compatible in the
projective system of covers.
\end{lemma}
\begin{proof}
Let $n$ be minimal with $X_n\neq \emptyset$. Choose a rigidified cover
of $X_n$. The support of $X_{n+1}\to X_n$ has only finitely many elements.
Choose a rigidified cover of $X_{n+1}$ compatible with all of them. Continue
inductively.

 Similar constructions show the rest of the assertion.
\end{proof}
\begin{defn}\label{chech}Let $X$ be a variety and  $\tilde{U}_X=\{U_i\}_{i\in I}$ a rigidified affine cover of $X$. We put
\[ C_\star(\tilde{U}_X)\in  C_-(\ZAff),\]
the \v{C}ech complex associated to the cover, i.e., 
\[ C_{n}(\tilde{U}_X)=\coprod_{\underline{i}\in I_n}\bigcap_{i\in\underline{i}} U_{i},\]
where $I_n$ is the set of tuples $(i_0,\dots,i_n)$. The boundary maps
are the ones obtained by taking the alternating sum of the boundary maps
of the simplicial scheme.

If $X_*\in C_b(\ZVar)$ is a complex, $\tilde{U}_{X_*}$ a rigidified affine cover, let 
\[ C_\star(\tilde{U}_{X_*})\in C_{-,b}(\ZAff)\]
be the double complex $C_i(\tilde{U}_{X_j})$.
\end{defn}

Note that all components of $C_\star(\tilde{U}_{X_*})$ are affine. The projective system
of these complexes is filtered and functorial. 

In the second step, we replace every affine $X$ by a complex of very good pairs.
This follows the key idea of Nori as follows:
Using induction one gets from the Basic Lemma \ref{basiclemma}:

\begin{cor}\label{goodfiltration}
Every affine variety $X$ has a filtration 
\[
\emptyset=F_{-1}X \subset F_0X \subset \cdots \subset F_{n-1}X \subset F_nX=X, 
\]
such that $(F_jX,F_{j-1}X,j)$ is very good.\end{cor}
Filtrations of the above type are called {\em very good filtrations}.
\begin{proof}
Let $\dim X=n$. Put $F_nX=X$. Choose a subvariety of dimension $n-1$ which
contains all singular points of $X$. By the Basic Lemma there is a
subvariety $F_{n-1}X$ of dimension $n-1$ such that $(F_nX,F_{n-1}X,n)$ is
good. By construction $F_{n-1}X\ohne F_{n-1}X$ is smooth and hence the
pair is very good. We continue by induction.
\end{proof}

\begin{cor}Let $X$ be an affine variety. The inductive system of all very good
filtrations of $X$ is filtered and functorial.
\end{cor}
\begin{proof}We follow Nori, \cite{N}. Let $F_*X$ and $F'_*X$ be two very good filtrations of $X$.
$F_ {n-1}X\cup F'_{n-1}X$ has dimension $n-1$. By the Basic Lemma
there is subvariety $G_{n-1}X\subset X$ of dimension
$n-1$ such that
$(X,G_{n-1}X,n)$ is a good pair. It is automatically very good.
We continue by induction.

Consider a morphism $f:X\to X'$. Let $F_*X$ be a very good filtration.
Then $f(F_{i}X)$ has dimension at most $i$. As in the proof of Corollary \ref{goodfiltration}, we construct a very
good filtration $F_*X'$ with additional property $f(F_iX)\subset F_iX'$.
\end{proof}

\begin{defn}Let $X$ be a variety, $\{U_i\}_{i\in I} $ a rigidified affine cover of $X$.
A {\em very good filtration} on $\tilde{U}_X$ is the choice of very good
filtrations for 
\[ \bigcap_{i\in J}U_i \]
for all $J\subset I$ compatible with all inclusions between these.

Let $f:X\to Y$ be a morphism of varieties, 
$\phi: \{U_i\}_{i\in I}\to\{V_j\}_{j\in J}$
a morphism of rigidified affine covers above $f$ . Fix very good filtrations
on both covers. $\phi$ is called {\em filtered}, if for all
$I'\subset I$ the induced map
\[ \bigcap_{i\in J}U_i\to \bigcap_{i\in I}V_{\phi(i)}\]
is compatible with the filtrations.

Let $X_*\in C_b(\ZVar)$ be bounded complex of varieties, $\tilde{U}_{X_*}$
an affine cover of $X_*$. A {\em very good filtration} on $\tilde{U}_{X_*}$
is a very good filtration on all $\tilde{U}_{X_n}$ compatible with all
morphisms in the support of the boundary maps. 
\end{defn}

Note that the \v{C}ech complex associated to a rigidified affine cover
with very good filtration
is also filtered in the sense that there is a very good
filtration on all $C_n(\tilde{U}_X)$ and all morphisms in the
support of the differential are compatible with the filtrations.

\begin{lemma}Let $X$ be a variety, $\tilde{U}_X$ a rigidified affine
cover. The inductive system of very good filtrations on
$\tilde{U}_X$ is non-empty, filtered and functorial.

The same statement also holds for a complex of varieties $X_*\in C_b(\ZVar)$.
\end{lemma}
\begin{proof} Let $\tilde{U}_X=\{U_i\}_{i\in I}$ be the affine
cover.
We chose recursively very good filtration on
$\bigcap_{i\in J}U_i$ with decreasing order of $J$, compatible
with the inclusions.

We extend the construction inductively to complexes, starting with the highest 
term of the complex.
\end{proof}

\begin{defn}Let $X_*\in C_-(\ZAff)$. A {\em very good filtration} of
$X_*$ is given by a very good filtration $F_. X_n$ for all $n$ which
is compatible with all morphisms in the support of the differentials
of $X_*$.
\end{defn}

\begin{lemma}Let $X_*\in C_b(\ZVar)$ and $\tilde{U}_{X_*}$ an affine 
cover of $X_*$ with a very good filtration. Then the total complex
of $C_\star(\tilde{U}_{X_*})$ carries a very good filtration.
\end{lemma}
\begin{proof}Clear by construction.
\end{proof}

Recall that $\tilde{D}^\eff\to \Ah$ be a representation is
of the diagram of very good pairs.

\begin{defn}Let $F_.X$ be an affine variety with a very good filtration.
We put $\tilde{R}(F_.X)\in C^b(\Ah)$
\[\dots\to T(F_jX_*,F_{j-1}X_*)\to T(F_{j+1}X_*,F_jX_*)\to \dots\]
Let $F_.X_*$ be a very good filtration of a complex $X_*\in C_-(\ZAff)$.
We put $\tilde{R}(F_.X_*)\in C^+(\Ah)$ the total complex of the
double complex $\tilde{R}(F_.X_n)_{n\in\Z}$.
\end{defn}

\begin{proof}[Proof of  Proposition \ref{complexrepr}:]
We first define $R:C^b(\ZVar)\to D^b(\Ah)$ on objects. Let
$X_*\in C_b(\ZVar)$. Choose a rigidified affine cover $\tilde{U}_{X_*}$ of $X_*$
. Choose a very good filtration on the cover. In induces
a very good filtration on
$\tot C_\star(\tilde{U}_{X_*})$.
Put
\[ R(X_*)=\tilde{R}(\tot C_\star(\tilde{U}_{X_*})).   \]
Note that any other choice yields a complex isomorphic to this one in
$D^+(\Ah)$. 

Let $f:X_*\to Y_*$ be a morphism. Choose a refinement $\tilde{U}'_{X_*}$ of $\tilde{U}_{X_*}$
which maps to $\tilde{U}_{Y_*}$ and a very good filtration on $\tilde{U}'_{X_*}$.
Choose a refinement of the filtrations on $\tilde{U}_{X_*}$ 
and $\tilde{U}_{Y_*}$
compatible
with the filtration on $\tilde{U}'_{X_*}$. This gives a little
diagram of morphisms of complexes $\tilde{R}$ which defines
$R(f)$ in $D^+(\Ah)$. 
\end{proof}

\begin{rem}Nori suggests working with Ind-objects (or rather pro-object in our dual setting) in order to get 
functorial complexes attached to affine varieties. However, the
mixing between inductive and projective systems in our construction does
not make it obvious if this works out for the result we needed. In order to avoid this situation, one could however do the
construction in two steps. 
\end{rem}
As a corollary of the construction in the proof, we also get:
\begin{cor}\label{cord.17} Let $X$ be a variety,  $\tilde{U}_{X}$ a rigidified
affine cover with \v{C}ech complex $C_\star(\tilde{U}_X)$. Then
\[ R(X)\to R(C_\star(\tilde{U}_X))\]
is an isomorphism in $D^+(\Ah)$.
\end{cor}

\end{document}